\documentclass[12pt, reqno, a4paper]{amsart}
\usepackage{amsmath,mathrsfs}
\usepackage{amssymb}
\usepackage{calligra}
\usepackage{dsfont}
\usepackage{bbm}
\usepackage{mathtools}
\usepackage{appendix}
\usepackage{autobreak}
\usepackage{enumitem}

\usepackage{pdfsync}

\newcommand\NoBlackBoxes{\global\overfullrule0pt}
\NoBlackBoxes

\newtheorem{definition}{Definition}[section]
\newtheorem{theorem}[definition]{Theorem}
\newtheorem{prop}[definition]{Proposition}
\newtheorem{lem}[definition]{Lemma}
\newtheorem{rem}[definition]{Remark}

\newtheorem{cor}[definition]{Corollary}

\newcommand{\1}{\mathbbm{1}}
\newcommand{\N}{{\mathbb{N}}}
\newcommand{\R}{{\mathbb{R}}}

\renewcommand{\P}{{\mathbb{P}}}
\newcommand{\E}{\mathbb{E}}
\newcommand{\V}{\mathbb{V}}
\newcommand{\CUn}{\mathcal{U}_n}
\newcommand{\vep}{\varepsilon}
\newcommand{\eps}{\varepsilon}
\newcommand{\nconv}{\xrightarrow[]{n\to\infty}}
\newcommand{\Pconv}{\xrightarrow[n\to\infty]{\P}}

\newcommand{\EW}[1]{\mathbb{E}\left[#1\right]}

\numberwithin{equation}{section}

\begin{document}
\title[A CLT for incomplete U-statistics]{A Central Limit Theorem for incomplete U-statistics over triangular arrays}

\author[Matthias L\"owe]{Matthias L\"owe}\thanks{Research of both authors was
funded by the Deutsche Forschungsgemeinschaft (DFG, German Research Foundation) under Germany's Excellence Strategy
EXC 2044-390685587, Mathematics M\"unster: Dynamics-Geometry-Structure}
\address[Matthias L\"owe]{Fachbereich Mathematik und Informatik,
Universit\"at M\"unster,
Einsteinstra\ss e 62,
48149 M\"unster,
Germany}

\email[Matthias L\"owe]{maloewe@uni-muenster.de}

\author[Sara Terveer]{Sara Terveer}
\address[Sara Terveer]{Fachbereich Mathematik und Informatik,
Universit\"at M\"unster,
Einsteinstra\ss e 62,
48149 M\"unster,
Germany}

\email[Sara Terveer]{sara.terveer@uni-muenster.de}



\begin{abstract}
We analyze the fluctuations of incomplete $U$-statistics over a triangular array of independent random variables.
We give criteria for a Central Limit Theorem (CLT, for short) to hold in the sense that we prove that an appropriately scaled and centered version of the U-statistic converges to a normal random variable. Our method of proof relies on a martingale CLT. A possible application -- a CLT for the hitting time for random walk on random graphs -- will be presented in \cite{LoTe20b}.
\end{abstract}

\maketitle
\section{Introduction}
U-statistics constitute a general method to construct unbiased minimum variance estimators in the theory of statistics. A thorough investigation of their properties can be found, e.g., in the monographs \cite{Denker85} or \cite{Lee90}.
U-statistics also naturally appear in other contexts, like in the theory of random graphs where they count occurrences of certain subgraphs, e.g.\ triangles (cf. \cite{Janson90}). In the latter case the U-statistics are incomplete, by which we mean that not all possible combinations of the random variables are taken into account. Such a ''dilution'' can also be random, as considered in \cite{Janson84}.
After having established a law of large numbers for U-statistics (cf.\ \cite{Christ_LLN}) the most obvious next question is to analyze their asymptotic distribution. This was already investigated in a seminal paper by Hoeffding \cite{Hoe48}. In general, whether a U-statistic is asymptotically normal or not, may depend on whether its kernel function is degenerate or not (\cite{Denker85}), i.e.\ on whether the conditional expectation of the kernel function given some the variables is zero or not.
Berry-Esseen theorems and Edgeworth expansions around this CLT were analyzed, among others, in \cite{BeG95} and \cite{BiGvZ86}.
Fluctuation results for U-statistics on the level of large or moderate deviations were obtained in \cite{EL95}, \cite{EL98}, and \cite{E98}.

In this note we will study a situation where the random variables in the U-statistic stem from a triangular array as in \cite{Mal87}. However, additionally to this we consider incomplete U-statistics, where a random variable determines whether a certain summand is taken into account or not. Finally, also the kernel function $h$ may change with $n$, the line number of the triangular array. This situation is motivated by our analysis of hitting times for random walks on random graphs in an accompanying paper (see \cite{LoTe20b}). However, as this situation also is a generalization of the settings in \cite{Janson84}, \cite{JJ86}, and \cite{Mal87}, we think it might be also interesting in its own rights.

To be more precise, let us describe our setting formally.
Let $X_{n1},\dots,X_{nn}, n=1, 2, \ldots, \infty$ be a triangular array of random variables with values in some measurable space (for the sake of this paper this measurable space will be $\R$), independent of each other and having the same distribution function $F_n(x)$ in each row.
Let
$h_n: \R \times \R \to \R$ be a real-valued, symmetric Borel function. For $i,j=1,\dots,n$, let
\[\Phi_n(i,j)=Z_{ij}\cdot h_n(X_{ni},X_{nj})\]
(of course, $\Phi_n$ is a function of $Z_{ij}, X_{ni}$, and $X_{nj}$ rather than just of $i,j$; however, for the sake of brevity we will omit the variables here and in the following definitions).
Here the $Z_{ij}=Z_{ji}$ are assumed to be i.i.d.\ $Ber(p_n)$ random variables (apart from the constraint that $Z_{ij}=Z_{ji}$) that are independent of the triangular array of the $(X_{ni})$. Throughout this note we will assume that $p$ may depend on $n$, but that $np \to \infty$.
Moreover, assume that for all $n\in \N$ and $1 \le i \neq j \le n$
\begin{equation}
\EW{h_n(X_{ni},X_{nj})}=0 \qquad \mbox{and } \quad \EW{h_n^2(X_{ni},X_{nj})}<\infty.
\label{eq:hmoments}
\end{equation}
Let us construct the following U-statistic
\[\mathcal{U}_n=\binom{n}{2}^{-1}\sum\limits_{1\leq i<j\leq n}\Phi_n(i,j)=\binom{n}{2}^{-1}\sum\limits_{1\leq i<j\leq n}Z_{i,j}\cdot h_n(X_{ni},X_{nj}).\]
To construct a Hoeffding-type decomposition (see \cite{Hoe48}) we introduce
\begin{align}
\begin{split}
\Psi_j^{(n)}(i)\coloneqq\EW{\Phi_n(i,j)\mid X_i,Z_{ij}}
=Z_{ij}\EW{h_n(X_{ni},X_{nj})\mid X_{ni}}\end{split}\label{Psi}\\
\begin{split}
\beta_n^2&\coloneqq\EW{\Phi_n^2(1,2)}, \quad \gamma_n^2\coloneqq\EW{\left(\Psi_2^{(n)}(1)\right)^2}, \mbox{ and } \theta_n^2\coloneqq np_n\gamma_n^2+\beta_n^2/2.
\label{defn}
\end{split}
\end{align}
Then obviously, $\Phi_n(i,j)$ and $\Psi_j^{(n)}(i)$ are centered. Next, put
\begin{align}
\tilde{\Phi}_n(i,j)&=\Phi_n(i,j)-\Psi_j^{(n)}(i)-\Psi_i^{(n)}(j),\notag\\
\tilde{h}_n(X_{ni},X_{nj})&=h_n(X_{ni},X_{nj})-\EW{h_n(X_{ni},X_{nj})\mid X_{ni}}-\EW{h_n(X_{ni},X_{nj})\mid X_{nj}}.\notag
\end{align}
Then $\tilde{\Phi}_n(i,j)=Z_{i,j}\tilde{h}_n(X_{ni},X_{nj})\label{eq:tildephi}$
is again centered -- we even have for every $i\neq j$ and every $k$
\begin{equation}
\EW{\tilde{h}_n(X_{ni},X_{nj})\mid X_{nk}}=0,
\label{eq:htildecentering}
\end{equation}
even for $k=i,j$ (as can be seen by applying the definition of $\tilde{h}_n$).

In the following we omit the index $n$ whenever suitable. We will also frequently write $h(i,j)$ and $\tilde{h}(i,j)$ as shorthand notations for $h_n(X_{ni},X_{nj})$ and $\tilde{h}_n(X_{ni},X_{nj})$.

Let us collect some properties of the above quantities in the following lemma, whose proof is deferred to the appendix:
\begin{lem}\label{lem:lem phi und h}
For any $i\neq j$ we have:
\begin{enumerate}
\item \label{properties:phi}
$\EW{\tilde{\Phi}^2(i,j)}=\beta_n^2-2\gamma_n^2$.
\item \label{properties:h}
$\EW{\EW{h(i,j)\mid X_i}^2}=\frac{\gamma_n^2}{p}, \mbox{ }\EW{h^2(i,j)}=\frac{\beta_n^2}{p} \mbox { and }
\EW{\tilde{h}(i,j)^2}=\frac{\beta_n^2-2\gamma_n^2}{p}.$
\end{enumerate}
\end{lem}

We can now compute the Hoeffing decomposition of $\mathcal{U}_n$ (again the proof is given in the appendix):
\begin{lem}\label{lem:hoeffing_dec}
We can rewrite $\CUn$ in the following way:
\begin{equation}
\CUn=\sum\limits_{i=1}^{n}\biggl(\binom{n}{2}^{-1}\sum\limits_{j=1}^{i-1}\tilde{\Phi}(i,j)+\binom{n}{2}^{-1}\sum_{\substack{j=1\\j\neq i}}^{n}\Psi_j(i)\biggr)
\label{eq:hoeffdingdecomposition}
\end{equation}

\end{lem}

This allows to compute the asymptotic variance of $\CUn$.
We will show in the appendix
\begin{lem}\label{lem:variance}
For the variance of $\CUn$ we have the following asymptotic identities
\[\V{\CUn}\sim\binom{n}{2}^{-1}\left(\beta_n^2+2np\gamma_n^2\right)=\binom{n}{2}^{-1}2\theta_n^2
\mbox{ and  }
\sqrt{\V{\CUn}}\binom{n}{2}\sim n\theta_n.\]
\end{lem}
Here and below, for two sequences $(a_n)$ and $(b_n)$ we write $a_n \sim b_n$, if $a_n/b_n \to 1$.

To prove a CLT for $\CUn$ we will consider $$\frac{\CUn}{\sqrt{\V {\CUn}}}\sim\sum\limits_{i=1}^{n}\xi_{i}$$
with $\xi_{i,n}=\xi_i=\xi_i^{(1)}+\xi_i^{(2)}$ for $i=1,\dots,n$, where
$$\xi_i^{(1)}=\xi_{i,n}^{(1)}=\frac{1}{n\theta_n}\sum_{\substack{j=1\\j\neq i}}^{n}\Psi_j(i),\quad \xi_i^{(2)}=\xi_{i,n}^{(2)}=\frac{1}{n\theta_n}\sum\limits_{j=1}^{i-1}\tilde{\Phi}(i,j).$$

We are aiming to prove the following results:

\begin{theorem}\label{theo1}
Assume that for all $\vep>0$
\begin{equation}	\eta_1=\sum\limits_{i=1}^{n}\E\bigl[\xi_i^2\1_{\{\left|\xi_i\right|\geq\eps\}}
\mid(X_k)_{k=1,\dots,i-1},(Z_{l,m})_{\substack{l=1,\dots,i-1,\phantom{nnl }\\m=1,\dots,n,\,m\neq l}}\bigr]\Pconv 0,
\label{16}\tag{B1}
\end{equation}
\begin{equation} \eta_2=\sum\limits_{i=1}^{n}\E\bigl[\xi_i^2\mid(X_k)_{k=1,\dots,i-1},
(Z_{l,m})_{\substack{l=1,\dots,i-1,\phantom{nnl}\\m=1,\dots,n,\,m\neq l}}\bigr]\Pconv 1.
\label{17}\tag{B2}
\end{equation}
Then $\frac{\CUn}{\sqrt{\V {\CUn}}}$ converges in distribution to a standard normal random variable.
\end{theorem}

To give alternative conditions that will be useful in the application we have in mind let us introduce
\begin{align}
{G}_k(i,j)&=\EW{{\Phi}(i,k){\Phi}(j,k)\mid X_i,X_j,Z_{ik},Z_{jk}}
\eqqcolon Z_{ik}Z_{jk}{H}(i,j)
\label{eq:G}\\
\tilde{G}_k(i,j)&=\EW{\tilde{\Phi}(i,k)\tilde{\Phi}(j,k)\mid X_i,X_j,Z_{ik},Z_{jk}}
\eqqcolon Z_{ik}Z_{jk}\tilde{H}(i,j).
\label{eq:tildeG}
\end{align}
Then,

\begin{theorem}\label{theo2}
Assume that for all $\vep>0$
\begin{align}
\frac{1}{n\theta_n^2}\E\Bigl[\bigl(\sum\limits_{j=2}^{n}\Psi_j(1)\bigr)^2\1_{\bigl\{\bigl|\sum\limits_{j=2}^{n}\Psi_j(1)\bigr|\geq\eps\theta_n n\bigr\}}\Bigr]&\nconv 0\label{6}\tag{C1}\\
\theta_n^{-2}\EW{\tilde{\Phi}^2(1,2)\1_{\left\{|\tilde{\Phi}(1,2)|\geq \eps\theta_n n\right\}}}&\nconv 0\label{20}\tag{C2}\\
p\,\theta_n^{-2}\E\bigl[\tilde{H}(1,1)\1_{\bigl\{|\tilde{H}(1,1)|\geq \frac{\eps\theta_n^2 n}{p}\bigr\}}\bigr]&\nconv 0\label{21}\tag{C3}\\
\theta_n^{-4}\EW{G_1^2(2,3)}&\nconv 0\label{23}\tag{C4}
\end{align}
Then again $\frac{\CUn}{\sqrt{\V {\CUn}}}$ converges in distribution to a standard normal random variable.
\end{theorem}


\begin{rem}
It is well known that there are situations where $\CUn$ does not obey a CLT. These are for example situations without dilution, e.g. when the $Z_{ij} \equiv 1$ with probability 1, and if the kernel function $h=h_n$ is degenerate, i.e. when
$\E[h(X,Y) \mid Y]=0$ for independent random variables $X,Y$ with the same distribution as $X_{11}$ (for the time being we assume that the distribution of $X_{ni}$ does not depend on $n$ and $i$, so we simply have a sequence of i.i.d. random variables
$X_1, X_2, \ldots$).

For a typical situations consider $\P(Z_{ij}=1)=1$ and $h(X_1,X_2)= X_1X_2$.
If then, $\E X_1=0$ and $\E X_1^2=1$, the rescaled U-statistic $n \CUn$ does not converge to a normally distributed random variable but to a $\chi_1^2$-random variable. This can be seen by the CLT together with an application of the continuous mapping theorem.
In this situation we quickly check that also condition \eqref{23} breaks down. Indeed, one checks that $\theta_n^2=\frac12$, since $\beta_n^2=1$, and that $\gamma_n^2=0$.
Moreover,
$$G_1(2,3)=H(2,3)=\EW{h(1,2)h(1,3)\mid X_2,X_3}=X_2X_3\EW{X_1^2}=X_2X_3$$
(recall that $Z_{i,j}\equiv 1$).
This means for \eqref{23} that
$$\theta_n^{-4}\EW{G_1^2(2,3)}=4\cdot\EW{X_2^2X_3^2}=4\EW{X_2^2}\EW{X_3^2}=4,$$
which does not go to $0$. Hence \eqref{23} is violated.
\end{rem}

The rest of this note is organized in the following way. In Section 2 we will prove Theorem \ref{theo1}. A main ingredient to this end will be a martingale CLT due to Girko (see Theorem \ref{thm:girkoneu} below). In Section 3 we will prepare for the proof of Theorem \ref{theo2} by giving an alternative condition for \eqref{23} (see condition \eqref{22} below). The core of the paper is the proof of Theorem \ref{theo2} in Section 4. We will see that \eqref{6}-\eqref{23} (resp. \eqref{22}) imply the conditions of Theorem \ref{theo1}. Finally, in Section 5, we will give some alternative conditions for \eqref{6}-\eqref{21} that are easier to check in some applications. The appendix contains the proofs of our technical results.

\section{Proof of Theorem \ref{theo1}}
As mentioned in the Introduction we will base our arguments on the following theorem (\cite[Theorem 5.4.11]{girko1990}):

\begin{theorem}\label{thm:girkoneu}
Consider a triangular array of martingale differences $(Y_{i,n})$, $i=1,\dots,n$ and a sequence of filtrations $(\mathcal{F}_{i,n})_{i=1,\dots,n}$. If for any $\vep>0$
$$\sum\limits_{i=1}^n\EW{Y_{i,n}^2\1_{\{\left|Y_{i,n}\right|\geq\eps\}}\mid\mathcal{F}_{i-1,n}}\Pconv 0 \mbox{ and }
\sum\limits_{i=1}^n\EW{Y_{i,n}^2\mid\mathcal{F}_{i-1,n}}\Pconv 1
$$
hold, then $\sum\limits_{i=1}^nY_{i,n}$ converges in distribution to a standard normal random variable.
\end{theorem}

To apply this result let
$\mathcal{F}_i=\mathcal{F}_{i,n}=\sigma\big((X_k)_{k=1,\dots,i},(Z_{l,m})_{\substack{l=1,\dots,i,\phantom{+m\neq l}\\m=1,\dots,n,\,m\neq l}}\big)$.
Using the notation from Section 1, $\Psi_j(i)$ is $\mathcal{F}_i$-measurable, hence $\xi_i^{(1)}$ is $\mathcal{F}_i$-adapted. For $j<i$, $X_j$ is also $\mathcal{F}_i$-measurable, so that $\tilde{\Phi}(i,j)$ and therefore $\xi_i^{(2)}$ are also $\mathcal{F}_i$-adapted. Hence $\xi_i$ is $\mathcal{F}_i$-adapted.
Now, 
\begin{align*}
\EW{\xi_i\mid\mathcal{F}_{i-1}}
&=\frac{1}{n\theta_n}\bigl(\sum\limits_{j< i}\EW{ Z_{ij}\EW{h(i,j)\mid X_i}\mid\mathcal{F}_{i-1}}+\sum\limits_{j> i}\EW{Z_{ij}\EW{h(i,j)\mid X_i}\mid\mathcal{F}_{i-1}}\bigr.\\
&\hspace{2cm}+\bigl.\sum\limits_{j<i}\EW{Z_{ij}\tilde{h}(i,j)\mid\mathcal{F}_{i-1}}\bigr).
\end{align*}
By definition, for $j<i$, $Z_{ij}=Z_{ji}$ are $\mathcal{F}_{i-1}$-measurable, while $X_i$ is independent of $\mathcal{F}_{i-1}$. For the second term notice that both $X_i$ and $Z_{ij}$ are independent of $\mathcal{F}_{i-1}$ in the case $j>i$. In the third sum, the $Z_{i,j}$ is measurable again. This leaves only $\tilde{h}(i,j)$, which is independent of all but one condition: $X_j$. Therefore
\begin{align*}
&\EW{\xi_i\mid\mathcal{F}_{i-1}}\\
&
=\frac{1}{n\theta_n}\left(\sum\limits_{j< i}Z_{ij}\EW{h(i,j)}+\sum\limits_{j>i}\EW{Z_{ij}}\EW{h(i,j)}+\sum\limits_{j<i}Z_{ij}\EW{\tilde{h}(i,j)\mid X_j}\right)=0
\end{align*}
by 
\eqref{eq:hmoments} and \eqref{eq:htildecentering}. Thus
$\xi_i$ is a martingale difference. Setting $Y_i=Y_{i,n}=\xi_{i,n}$ in Theorem \ref{thm:girkoneu} we can rewrite
conditions a) and b) in this theorem as
\begin{equation*}	\eta_1=\sum\limits_{i=1}^{n}\E\bigl[\xi_i^2\1_{\{\left|\xi_i\right|\geq\eps\}}
\mid(X_k)_{k=1,\dots,i-1},(Z_{l,m})_{\substack{l=1,\dots,i-1,\phantom{nnl}\\m=1,\dots,n,\,m\neq l}}\bigr]\Pconv 0,
\end{equation*}
\begin{equation*} \eta_2=\sum\limits_{i=1}^{n}\E\bigl[\xi_i^2\mid(X_k)_{k=1,\dots,i-1},
(Z_{l,m})_{\substack{l=1,\dots,i-1,\phantom{nnl}\\m=1,\dots,n,\,m\neq l}}\bigr]\Pconv 1.
\end{equation*}
This proves Theorem \ref{theo1}.

\section{An alternative condition for \eqref{23}}
The purpose of this section is to prove
\begin{prop}\label{lem:23to24}
\eqref{23} implies
\begin{equation}
\theta_n^{-4}\EW{\tilde{G}_1^2(2,3)}\nconv 0.\label{22}\tag{C4'}
\end{equation}
\end{prop}

\begin{proof}
For $i,j,k$ pairwise different, we have by definition of $\tilde G_k$
	\begin{align}
	&\tilde{G}_k(i,j)
	=\mathbb{E}\bigg[\Phi(i,k)\Phi(j,k)-\Psi_k(i)\Phi(j,k)-\Psi_i(k)\Phi(j,k)-\Psi_k(j)\Phi(i,k)-\Psi_j(k)\Phi(i,k)\notag\\ &\phantom{=E}\left.+\Psi_k(i)\Psi_k(j)+\Psi_k(i)\Psi_j(k)\right.
+\Psi_i(k)\Psi_k(j)+\Psi_i(k)\Psi_j(k)\mid X_i,X_j,Z_{i,k},Z_{j,k}\bigg]\notag\\
	=&G_k(i,j)-\Psi_k(i)\Psi_k(j)-\EW{\Psi_i(k)\Phi(j,k)\mid X_i,X_j,Z_{i,k},Z_{j,k}}-\Psi_k(i)\Psi_k(j)\notag\\
	&-\EW{\Psi_j(k)\Phi(i,k)\mid X_i,X_j,Z_{i,k},Z_{j,k}}
+\Psi_k(i)\Psi_k(j)+\Psi_k(i)\EW{\Psi_j(k)\mid X_i,X_j,Z_{i,k},Z_{j,k}}\notag\\
	&+\Psi_k(j)\EW{\Psi_i(k)\mid X_i,X_j,Z_{i,k},Z_{j,k}}+\EW{\Psi_i(k)\Psi_j(k)\mid X_i,X_j,Z_{i,k},Z_{j,k}}\notag
	\end{align}

	Now $\EW{\Psi_j(k)\mid X_i,X_j,Z_{i,k},Z_{j,k}}=0$, and three of the above terms only differ by their sign. Thus
	\begin{align}
	&\tilde{G}_k(i,j)	
	= G_k(i,j)-\Psi_k(i)\Psi_k(j)-\EW{\Psi_i(k)\Phi(j,k)\mid  X_i,X_j,Z_{i,k},Z_{j,k}}\notag\\
	&\phantom{=E}-\EW{\Psi_j(k)\Phi(i,k)\mid  X_i,X_j,Z_{i,k},Z_{j,k}}+\EW{\Psi_i(k)\Psi_j(k)\mid  X_i,X_j,Z_{i,k},Z_{j,k}}\notag\\
	&\eqqcolon G_k(i,j)-\Psi_k(i)\Psi_k(j)-A-B+C.\label{tildegsep}
	\end{align}
	By independence, $\EW{\left(\Psi_k(i)\Psi_k(j)\right)^2}=\EW{\Psi_k^2(i)}\EW{\Psi_k^2(j)}=\gamma_n^4$, and
 for $A$, we have again by independence and Cauchy-Schwarz
	\begin{align*}
	\EW{A^2}&=\EW{Z_{i,k}Z_{j,k}\left(\EW{\EW{h(i,k)\mid X_k}h(j,k)\mid X_j}\right)^2}
	\\&
	\leq p^2\EW{\EW{\EW{h(i,k)\mid X_k}^2\mid X_j}\EW{h(j,k)^2\mid X_j}}
	\\&
	= p^2\EW{\EW{\EW{h(i,k)\mid X_k}^2}\EW{h^2(j,k)\mid X_j}}\\
	\intertext{and by Lemma \ref{lem:lem phi und h} this equals}
	&= p^2\EW{\frac{\gamma_n^2}{p}\EW{h^2(j,k)\mid X_j}}= p^2\frac{\gamma_n^2}{p}\EW{h^2(j,k)}
	= p^2\frac{\gamma_n^2}{p}\frac{\beta_n^2}{p}=\beta_n^2\gamma_n^2.
	\end{align*}
	$B^2$ has the same upper bound for the expectation, which can be proven analogously. As for $C$, we  again use measurability and independence to obtain
	\begin{align*}
	\EW{C^2}&=\EW{\EW{\Psi_i(k)\Psi_j(k)\mid  X_i,X_j,Z_{i,k},Z_{j,k}}^2}
	=\EW{\EW{\Psi_i(k)\Psi_j(k)\mid  Z_{i,k},Z_{j,k}}^2}\\
	&=\EW{Z_{i,k}Z_{j,k}\EW{\EW{h(i,k)\mid X_k}\EW{h(j,k)\mid X_k}}^2}\\
	&\leq p^2\EW{\EW{h(i,k)\mid X_k}^2}^2=p^2\left(\frac{\gamma_n^2}{p}\right)^2=\gamma_n^4
	\end{align*}
	by Cauchy-Schwarz and Lemma \ref{lem:lem phi und h}.
	
	We can combine all this to conclude
	\begin{align*}	\frac{1}{\theta_n^4}\EW{\tilde{G}_1^2(2,3)}&\leq\frac{1}{\theta_n^4}\EW{\left(G_1(2,3)-\Psi_1(2)\Psi_1(3)-A-B+C\right)^2}\\
	&\leq \frac{25}{\theta_n^4}\left(\EW{G_1^2(2,3)}+\EW{\Psi_1^2(2)\Psi_1^2(3)}+\EW{A^2}+\EW{B^2}+\EW{C^2}\right)\\
	&\leq \frac{25}{\theta_n^4}\EW{G_1^2(2,3)}+\frac{25}{\theta_n^4}2\gamma_n^4+\frac{25}{\theta_n^4}2\beta_n^2\gamma_n^2
\leq \frac{25}{\theta_n^4}\EW{G_1^2(2,3)}+\frac{50}{(np)^2}+\frac{50}{np}
	\end{align*}
due to \eqref{defn}, $\theta_n^2\geq np\gamma_n^2$ and $\theta_n^2\geq\beta_n^2$.
	The second and third term go to 0 as $np\to\infty$ for $n\to\infty$. The first term is exactly the term from \eqref{23}, which is assumed to converge to 0 as well. This completes the proof.
\end{proof}

\section{Proof of Theorem \ref{theo2}}
The proof of Theorem \ref{theo2} immediately follows from
\begin{prop}\label{prop:cond1}
\eqref{16} and \eqref{17} follow from \eqref{6}, \eqref{20}, \eqref{21}, and \eqref{23}.

%
\end{prop}

\begin{rem}
Conditions \eqref{6}--\eqref{23} may be tricky to check. In fact, in many settings it may be unreasonable to prove conditions for $\tilde{\Phi}$ instead of $\Phi$ etc. In Proposition \ref{prop:altcond} below we will give alternative, more straightforward conditions for \eqref{6}--\eqref{21}. 
\end{rem}

We split the proof of Proposition \ref{prop:cond1} into two Lemmas:

\begin{lem}
\eqref{16} follows from \eqref{6}, \eqref{20}, \eqref{21}, \eqref{23} and \eqref{22}.
\end{lem}
\begin{proof}
Since $\eta_1$ is non-negative, by Markov's inequality, it suffices to show $\EW{\eta_1}\nconv 0$ to obtain $\eta_1\Pconv 0$.
Using Lemma \ref{lem:quadratabsch} with $k=2$, $a_1=\xi_i^{(1)}$, and $a_2=\xi_i^{(2)}$ we get
\begin{align*}\eta_1
&\leq 4\sum\limits_{i=1}^{n}\left(\EW{\left(\xi_i^{(1)}\right)^2\1_{\left\{\left|\xi_i^{(1)}\right|\geq\frac{\eps}{2}\right\}}\mid(X_k)_{k=1,\dots,i-1},(Z_{l,m})_{\substack{l=1,\dots,i-1,\phantom{nnl}\\m=1,\dots,n,\,m\neq l}}}\right.\\
&\phantom{=E}+\EW{\left.\left(\xi_i^{(2)}\right)^2\1_{\left\{\left|\xi_i^{(2)}\right|\geq\frac{\eps}{2}\right\}}\mid(X_k)_{k=1,\dots,i-1},(Z_{l,m})_{\substack{l=1,\dots,i-1,\phantom{nnl}\\m=1,\dots,n,\,m\neq l}}}\right),
\end{align*}
and consequently,
\begin{align}
\begin{split}
\EW{\eta_1}
&\leq 4\sum\limits_{i=1}^{n}\E\Bigl[\bigl(\frac{1}{n\theta_n}\sum_{\substack{j=1\\j\neq i}}^{n}\Psi_j(i)\bigr)^2\1_{\bigl\{\bigl|\frac{1}{n\theta_n}\sum\limits_{j=1,j\neq i}^{n}\Psi_j(i)\bigr|\geq\frac{\eps}{2}\bigr\}}\Bigr]\\
&\phantom{=E}
+4\sum\limits_{i=1}^{n}\E\Bigl[\bigl(\frac{1}{n\theta_n}\sum\limits_{j=1}^{i-1}\tilde{\Phi}(i,j)\bigr)^2\1_{\bigl\{\bigl|\frac{1}{n\theta_n}
\sum\limits_{j=1}^{i-1}\tilde{\Phi}(i,j)\bigr|\geq\frac{\eps}{2}\bigr\}}\Bigr]
\eqqcolon S_1+T_1.
\label{eweta1}
\end{split}
\end{align}
Now,
\begin{align*}
S_1&=\frac{4}{n^2\theta_n^2}\sum\limits_{i=1}^{n}\E \Bigl[\bigl(\sum_{j=1, j\neq i}^{n}\Psi_j(i)\bigr)^2\1_{\bigl\{\bigl|\sum\limits_{j=1,j\neq i}^{n}\Psi_j(i)\bigr|\geq\frac{\eps\theta_n n}{2}\bigr\}}\Bigr]
=\frac{4}{n\theta_n^2}\E\Bigl[\bigl(\sum_{j=2}^{n}\Psi_j(1)\bigr)^2\1_{\bigl\{\bigl|\sum\limits_{j=2}^{n}\Psi_j(1)\bigr|\geq\frac{\eps\theta_n n}{2}\bigr\}}\Bigr]
\end{align*}
by identical distribution. Hence, by \eqref{6}, we have $S_1\nconv 0.$

The estimate for $T_1$ is slightly longer:
\begin{align*}
T_1
&=\frac{4}{n^2\theta_n^2}\sum\limits_{i=1}^{n}\E\Bigl[\bigl(\sum\limits_{j=1}^{i-1}\tilde{\Phi}(i,j)\bigr)^2
\1_{\bigl\{\bigl|\sum\limits_{j=1}^{i-1}\tilde{\Phi}(i,j)\bigr|\geq\frac{\eps\theta_n n}{2}\bigr\}}\Bigr]\\
&=\frac{4}{n^2\theta_n^2}\sum\limits_{i=1}^{n}\sum\limits_{j=1}^{i-1}
\E \Bigl[ \tilde{\Phi}^2(i,j)\1_{\bigl\{\bigl|\sum\limits_{k=1}^{i-1}\tilde{\Phi}(i,k)\bigr|\geq\frac{\eps\theta_n n}{2}\bigr\}}\1_{\bigl\{\bigl|\tilde{\Phi}(i,j)\bigr|\geq\frac{\eps\theta_n n}{4}\bigr\}}\Bigr]\\
&\phantom{=E}+\frac{4}{n^2\theta_n^2}\sum\limits_{i=1}^{n}\sum\limits_{j=1}^{i-1}\E\Bigl[\tilde{\Phi}^2(i,j)
\1_{\bigl\{\bigl|\sum\limits_{k=1}^{i-1}\tilde{\Phi}(i,k)\bigr|\geq\frac{\eps\theta_n n}{2}\bigr\}}\1_{\bigl\{\bigl|\tilde{\Phi}(i,j)\bigr|<\frac{\eps\theta_n n}{4}\bigr\}}\Bigr]\\
&\phantom{=E}+\frac{4}{n^2\theta_n^2}\sum\limits_{i=1}^{n}\Bigl|\E\Bigl[\sum\limits_{1\le j\neq k \le i-1} \tilde{\Phi}(i,j)\tilde{\Phi}(i,k)\1_{\bigl\{\bigl|\sum\limits_{j=1}^{i-1}\tilde{\Phi}(i,j)\bigr|\geq\frac{\eps\theta_n n}{2}\Bigr\}}\Bigr]\Bigr|\\
&\leq\frac{4}{\theta_n^2}\EW{\tilde{\Phi}^2(1,2)\1_{\left\{\left|\tilde{\Phi}(1,2)\right|\geq\frac{\eps\theta_n n}{4}\right\}}}
+\frac{4}{n^2\theta_n^2}\sum\limits_{1\le j<i\le n}
\E\Bigl[\tilde{\Phi}^2(i,j)
\1_{\bigl\{\bigl|\sum\limits_{k=1}^{i-1}\tilde{\Phi}(i,k)\bigr|\geq\frac{\eps\theta_n n}{2}\bigl\}}\1_{\bigl\{\bigl|\tilde{\Phi}(i,j)\bigr|<\frac{\eps\theta_n n}{4}\bigr\}}\Bigr]
\\&\phantom{=E}
+\frac{4}{n^2\theta_n^2}\sum\limits_{i=1}^{n}\Bigl|\E\Bigl[\sum\limits_{1\le j\neq k\le i-1} \tilde{\Phi}(i,j)\tilde{\Phi}(i,k)\1_{\bigl\{\bigl|\sum\limits_{j=1}^{i-1}\tilde{\Phi}(i,j)\bigr|\geq\frac{\eps\theta_n n}{2}\bigr\}} \Bigr]\Bigr|\\
&\eqqcolon S_2+S_3+S_4
\end{align*}
By \eqref{20}, we find that
$S_2\nconv 0.$
For $S_3$, we manipulate the indicators to see that
$$
\1_{\bigl\{\bigl|\sum\limits_{k=1}^{i-1}\tilde{\Phi}(i,k)\bigr|\geq\frac{\eps\theta_n n}{2}\bigl\}}\1_{\bigl\{\bigl|\tilde{\Phi}(i,j)\bigr|<\frac{\eps\theta_n n}{4}\bigr\}} \le
\1_{\bigl\{\bigl|\sum\limits_{1 \le k \le i-1 \atop k\neq j}\tilde{\Phi}(i,k)\bigr|\geq\frac{\eps\theta_n n}{4}\bigr\}}
$$
and use \eqref{eq:tildeG} to obtain
\begin{align*}
S_3&\leq
\frac{4}{n^2\theta_n^2}\sum\limits_{1\le j< i\le n}
\E\Bigl[\E \Bigl[\tilde{\Phi}^2(i,j)\1_{\bigl\{\bigl|\sum\limits_{1 \le k \le i-1 \atop k\neq j}\tilde{\Phi}(i,k)\bigr|\geq\frac{\eps\theta_n n}{4}\bigr\}}\mid(X_k)_{\substack{k=1,\dots,i\\k\neq j\phantom{,\dots,i}}},(Z_{i,k})_{k=1,\dots,i-1}\Bigr]\Bigr]\\
&=\frac{4}{n^2\theta_n^2}\sum\limits_{1\le j< i\le n}\E\Bigl[\E\Bigl[\tilde{\Phi}^2(i,j)\mid X_i,Z_{i,j}\Bigr]\1_{\bigl\{\bigl|\sum\limits_{1 \le k \le i-1 \atop k\neq j}\tilde{\Phi}(i,k)\bigr|\geq\frac{\eps\theta_n n}{4}\bigr\}}
\Bigr]\\
&=\frac{4}{n^2\theta_n^2}\sum\limits_{1\le j<i}^{n}\E\Bigl[\tilde{G}_j(i,i)\1_{\bigl\{\bigl|\sum\limits_{1 \le k \le i-1 \atop k\neq j}\tilde{\Phi}(i,k)\bigr|\geq\frac{\eps\theta_n n}{4}\bigr\}}\Bigr]
\end{align*}
By adding another indicator, for any $\tilde{\eps}>0$,  this can be rewritten as
\begin{align*}
S_3&\le \frac{4}{n^2\theta_n^2}\sum\limits_{1\le j< i\le n}\E\Bigl[\tilde{G}_j(i,i)\1_{\bigl\{\bigl|\sum\limits_{1 \le k \le i-1 \atop k\neq j}\tilde{\Phi}(i,k)\bigr|\geq\frac{\eps\theta_n n}{4}\bigl\}}\1_{\bigl\{|\tilde{H}(i,i)|\geq\frac{\tilde{\eps}\theta_n^2 n}{16p}\bigr\}}\Bigr]\\
&\phantom{=E}+\frac{4}{n^2\theta_n^2}\sum\limits_{1\le j< i\le n}\E\Bigl[\tilde{G}_j(i,i)\1_{\bigl\{\bigl|\sum\limits_{1 \le k \le i-1 \atop k\neq j}\tilde{\Phi}(i,k)\bigr|\geq\frac{\eps\theta_n n}{4}\bigr\}}\1_{\bigl\{|\tilde{H}(i,i)|<\frac{\tilde{\eps}\theta_n^2 n}{16p}\bigr\}}\Bigr]
\end{align*}
\begin{align*}
&\le \frac{4}{n^2\theta_n^2}\sum\limits_{1\le j< i\le n}\E\Bigl[\tilde{G}_j(i,i)\1_{\left\{|\tilde{H}(i,i)|\geq\frac{\tilde{\eps}\theta_n^2 n}{16p}\right\}}\Bigr]\\
&\phantom{=E}+\frac{4}{n^2\theta_n^2}\sum\limits_{1\le j< i\le n}\E\Bigl[\tilde{G}_j(i,i)\1_{\bigl\{\bigr|\sum\limits_{1 \le k \le i-1 \atop k\neq j}\tilde{\Phi}(i,k)\bigr|\geq\frac{\eps\theta_n n}{4}\bigr\}}\1_{\bigl\{|\tilde{H}(i,i)|<\frac{\tilde{\eps}\theta_n^2 n}{16p}\bigr\}}\Bigr]\\
&\eqqcolon S_{31}+S_{32}.
\end{align*}
For $S_{31}$ we have by independence, \eqref{eq:tildeG} and \eqref{21}
\begin{align}\begin{split}\label{eq:s31}
S_{31}
&=\frac{4}{n^2\theta_n^2}\binom{n}{2}p\EW{\tilde{H}(1,1)\1_{\left\{|\tilde{H}(1,1)|\geq\frac{\tilde{\eps}\theta_n^2 n}{16p}\right\}}}\nconv 0.
\end{split}
\end{align}
On the other hand, by applying the two indicators
\begin{align*}
S_{32}
&\leq \frac{4}{n^2\theta_n^2}\sum\limits_{1\le j<i\le n }\E\Bigl[Z_{i,j}\frac{\tilde{\eps}\theta_n^2 n}{16p}\1_{\bigl\{\bigl|\sum\limits_{1 \le k \le i-1 \atop k\neq j}\tilde{\Phi}(i,k)\bigr|\geq\frac{\eps\theta_n n}{4}\bigr\}}\Bigr]
\\&\leq \frac{4\tilde{\eps}}{\eps^2n^3\theta_n^2}\sum\limits_{1\le j<i\le n }\E\Bigl[\frac{\eps^2\theta_n^2 n^2}{16}\1_{\bigl\{\bigl|\sum\limits_{1 \le k \le i-1 \atop k\neq j}\tilde{\Phi}(i,k)\bigr|\geq\frac{\eps\theta_n n}{4}\bigr\}}\Bigr]\\
&\leq \frac{4\tilde{\eps}}{\eps^2n^3\theta_n^2}\sum\limits_{1\le j<i\le n }\E\Bigl[\Bigl(\Bigl|\sum\limits_{1 \le k \le i-1 \atop k\neq j}\tilde{\Phi}(i,k)\Bigr|\Bigr)^2\1_{\bigl\{\bigl|\sum\limits_{1 \le k \le i-1 \atop k\neq j}\tilde{\Phi}(i,k)\bigr|\geq\frac{\eps\theta_n n}{4}\bigr\}}\Bigr]\\
&\leq \frac{4\tilde{\eps}}{\eps^2n^3\theta_n^2}\sum\limits_{1\le j<i\le n}\E\Bigl[\sum\limits_{1 \le k \le i-1 \atop k\neq j}\tilde{\Phi}^2(i,k)+\sum\limits_{1 \le k \le i-1 \atop k\neq j}\sum\limits_{1\le l \le i-1 \atop l\neq j,k}\tilde{\Phi}(i,k)\tilde{\Phi}(i,l)\Bigr]\\
&\leq \frac{4\tilde{\eps}}{\eps^2n^3\theta_n^2}n(n-1)\left((n-1)\EW{\tilde{\Phi}^2(1,2)}+(n-1)^2\EW{\tilde{\Phi}(1,2)\tilde{\Phi}(1,3)}\right)\\
\end{align*}
Applying Lemma \ref{lem:lem phi und h}, the first expectation is smaller than $\beta_n^2$.
By Lemma \ref{lem:phikleinindep}, the second expectation is 0.
Thus
$$ S_{32} 
\le \frac{4(n-1)^2\tilde{\eps}}{\eps^2n^2\theta_n^2}\beta_n^2 \le \frac{8\tilde{\eps}}{\eps^2}\frac{(n-1)^2}{n^2}
$$
by \eqref{defn}.  
As we may chose $\tilde{\eps}>0$ arbitrarily, 
this shows that $S_{32}\to 0$, and together with \eqref{eq:s31} we obtain that
$S_3\nconv 0.$

Finally, for $S_4$, we compute
\begin{align*}
S_4
&\leq\frac{4}{n^2\theta_n^2}\sum\limits_{i=1}^{n}\Bigl|\E\Bigl[\1_{\{|\tilde{H}(i,i)|\geq \frac{\tilde{\eps}\theta_n^2n}{2p}\}}\sum\limits_{1 \le j,k \le i-1 \atop j\neq k}^{i-1}\tilde{\Phi}(i,j)\tilde{\Phi}(i,k)\1_{\bigl\{\bigl|\sum\limits_{j=1}^{i-1}\tilde{\Phi}(i,j)\bigr|\geq\frac{\eps\theta_n n}{2}\bigr\}}\Bigr]\Bigr|\\
&\hspace{2cm}+\frac{4}{n^2\theta_n^2}\sum\limits_{i=1}^{n}\Bigl|\E\Bigl[\1_{\{|\tilde{H}(i,i)|< \frac{\tilde{\eps}\theta_n^2n}{2p}\}}\sum\limits_{1 \le j,k \le i-1 \atop j\neq k}\tilde{\Phi}(i,j)\tilde{\Phi}(i,k)\1_{\bigl\{\bigl|\sum\limits_{j=1}^{i-1}\tilde{\Phi}(i,j)\bigr|\geq\frac{\eps\theta_n n}{2}\bigr\}}\Bigr]\Bigr|\\
&\eqqcolon S_{41}+S_{42}
\end{align*}
Consider $S_{41}$ first. Because of $
\Bigl|\sum\limits_{j\neq k}x_jx_k\Bigr|
\leq \Bigl(\sum\limits_{j=1}^nx_j\Bigr)^2+\sum\limits_{j=1}^nx_j^2,$
we obtain
\begin{align*}
S_{41}
&\leq \frac{4}{n^2\theta_n^2}\sum\limits_{i=1}^{n}\E\Bigl[\1_{\{|\tilde{H}(i,i)|\geq \frac{\tilde{\eps}\theta_n^2n}{2p}\}}\Bigl(\sum\limits_{j=1}^{i-1}\tilde{\Phi}(i,j)\Bigr)^2\1_{\bigl\{\bigl|\sum\limits_{j=1}^{i-1}\tilde{\Phi}(i,j)\bigr|\geq\frac{\eps\theta_n n}{2}\bigr\}}\Bigr]\\
&\hspace{2cm}+ \frac{4}{n^2\theta_n^2}\sum\limits_{i=1}^{n}\E\Bigl[\1_{\{|\tilde{H}(i,i)|\geq \frac{\tilde{\eps}\theta_n^2n}{2p}\}}\sum\limits_{j=1}^{i-1}\tilde{\Phi}^2(i,j)\1_{\bigl\{\bigl|\sum\limits_{j=1}^{i-1}\tilde{\Phi}(i,j)\bigr|\geq\frac{\eps\theta_n n}{2}\bigr\}}\Bigr]\\
&\leq \frac{8}{n^2\theta_n^2}\sum\limits_{i=1}^{n}\E\Bigl[\1_{\{|\tilde{H}(i,i)|\geq \frac{\tilde{\eps}\theta_n^2n}{2p}\}}\sum\limits_{j=1}^{i-1}\tilde{\Phi}^2(i,j)\Bigr]
+\frac{4}{n^2\theta_n^2}\sum\limits_{i=1}^{n}\E\Bigl[\1_{\{|\tilde{H}(i,i)|\geq \frac{\tilde{\eps}\theta_n^2n}{2p}\}}\sum\limits_{1 \le j,k \le i-1  \atop j\neq k}\tilde{\Phi}(i,j)\tilde{\Phi}(i,k)\Bigr]\\
&\eqqcolon S_{41}'+S_{41}''
\end{align*}
For the first of these summands we obtain
\begin{align*}
S_{41}'
&=\frac{8}{n^2\theta_n^2}\sum\limits_{1\le j <i\le n}\E\Bigl[\1_{\{|\tilde{H}(i,i)|\geq \frac{\tilde{\eps}\theta_n^2n}{2p}\}}\tilde{\Phi}^2(i,j)\Bigr] 
=\frac{8}{n^2\theta_n^2}\sum\limits_{1\le j <i\le n}\E\Bigl[\1_{\{|\tilde{H}(i,i)|\geq \frac{\tilde{\eps}\theta_n^2n}{2p}\}}\E\Bigl[\tilde{\Phi}^2(i,j)\mid X_i, Z_{ij}\Bigr]\Bigr]\\
&=\frac{8}{n^2\theta_n^2}\sum\limits_{1\le j <i\le n}\E\Bigl[\1_{\{|\tilde{H}(i,i)|\geq \frac{\tilde{\eps}\theta_n^2n}{2p}\}}\tilde{G}_j(i,i)\Bigr]
=\frac{8}{n^2\theta_n^2}\binom{n}{2}\E\Bigl[\1_{\{|\tilde{H}(1,1)|\geq \frac{\tilde{\eps}\theta_n^2n}{2p}\}}\tilde{H}(1,1)Z_{1,2}\Bigr]\\
&
\leq\frac{4}{\theta_n^2}p\E\Big[\1_{\{|\tilde{H}(1,1)|\geq \frac{\tilde{\eps}\theta_n^2n}{2p}\}}\tilde{H}(1,1)\Bigr]
\end{align*}
which leads to $S_{41}'\nconv 0$
due to \eqref{21}.

For the second summand $S_{41}''$ we estimate
\begin{align*}
S_{41}''
&=\frac{4}{n^2\theta_n^2}\sum\limits_{i=1}^{n}\sum\limits_{1 \le j,k \le i-1 \atop j\neq k}\E\Bigl[\E\Bigl[\1_{\{|\tilde{H}(i,i)|\geq \frac{\tilde{\eps}\theta_n^2n}{2p}\}}\tilde{\Phi}(i,j)\tilde{\Phi}(i,k)\mid X_i,X_j,Z_{i,j},Z_{i,k}\Bigr]\Bigr]\\
&=\frac{4}{n^2\theta_n^2}\sum\limits_{i=1}^{n}\sum\limits_{1 \le j,k \le i-1 \atop j\neq k}\E\Bigl[\1_{\{|\tilde{H}(i,i)|\geq \frac{\tilde{\eps}\theta_n^2n}{2p}\}}\tilde{\Phi}(i,j)Z_{i,k}\E\bigl[\tilde{h}(i,k)\mid X_i,X_j,Z_{i,j},Z_{i,k}\bigr]\Bigr]\\
&=\frac{4}{n^2\theta_n^2}\sum\limits_{i=1}^{n}\sum\limits_{1 \le j,k \le i-1 \atop j\neq k}\E\Bigl[\1_{\{|\tilde{H}(i,i)|\geq \frac{\tilde{\eps}\theta_n^2n}{2p}\}}\tilde{\Phi}(i,j)Z_{i,k}\E\bigl[\tilde{h}(i,k)\mid X_i\bigr]\Bigr]=0,
\end{align*}
which follows from \eqref{eq:htildecentering}. Altogether this gives
$S_{41}\nconv 0.$

Considering now $S_{42}$ we see that (using Cauchy-Schwarz for the second inequality)
\begin{align}
S_{42}
&\leq\frac{4}{n^2\theta_n^2}\sum\limits_{i=1}^{n}\E\Bigl[\Bigl|\1_{\{|\tilde{H}(i,i)|< \frac{\tilde{\eps}\theta_n^2n}{2p}\}}\sum\limits_{1 \le j,k \le i-1 \atop j\neq k}\tilde{\Phi}(i,j)\tilde{\Phi}(i,k)\1_{\bigl\{\bigl|\sum\limits_{j=1}^{i-1}\tilde{\Phi}(i,j)\bigr|\geq\frac{\eps\theta_n n}{2}\bigr\}}\Bigr|\Bigr]\notag\\
&\leq \frac{4}{n^2\theta_n^2}\sum\limits_{i=1}^{n}\E\Bigl[\1_{\{|\tilde{H}(i,i)|<\frac{\tilde{\eps}\theta_n^2n}{2p}\}}\bigl(\sum\limits_{j,k=1}^{i-1}\1_{\left\{j\neq k\right\}}\tilde{\Phi}(i,j)\tilde{\Phi}(i,k)\bigr)^2\Bigr]^{\frac{1}{2}}\P\Bigl(
\bigl|\sum\limits_{l=1}^{i-1}\tilde{\Phi}(i,l)\bigr|\geq\frac{\eps\theta_nn}{2}\Bigr)^{\frac{1}{2}}\notag\\
&=:\frac{4}{n^2\theta_n^2}\sum\limits_{i=1}^{n}A_iB_i.\label{28}
\end{align}
We estimate $A_i$ by
\begin{align*}
A_i^2
&=\sum_{j,k=1}^{i-1}\1_{\{j\neq k\}}\E\bigl[\1_{\{|\tilde{H}(i,i)|<\frac{\tilde{\eps}\theta_n^2n}{2p}\}}\tilde{\Phi}^2(i,j)\tilde{\Phi}^2(i,k)\bigr]\\
&\hspace{1cm}+\underset{\left\{l,m\right\}\neq \left\{j,k\right\}}{\sum\limits_{j,k=1}^{i-1}\sum\limits_{l,m=1}^{i-1}}\1_{\{j\neq k\}}\1_{\{l\neq m\}}\E\bigl[\1_{\{|\tilde{H}(i,i)|<\frac{\tilde{\eps}\theta_n^2n}{2p}\}}\tilde{\Phi}(i,j)\tilde{\Phi}(i,k)\tilde{\Phi}(i,l)\tilde{\Phi}(i,m)\Bigr]\\
&\eqqcolon A_{i1}+A_{i2}
\end{align*}
Using the properties of (conditional) expectation for $A_{i1}$ we see that
\begin{align*}
\E\bigl[\1_{\{|\tilde{H}(i,i)|<\frac{\tilde{\eps}\theta_n^2n}{2p}\}}\tilde{\Phi}^2(i,j)\tilde{\Phi}^2(i,k)\bigr]
\hspace{-5cm}&\\
&=\E\Bigl[\E\Bigl[\E\Bigl[\1_{\{|\tilde{H}(i,i)|<\frac{\tilde{\eps}\theta_n^2n}{2p}\}}\tilde{\Phi}^2(i,j)\tilde{\Phi}^2(i,k)\mid X_i,X_j,Z_{i,j}, Z_{i,k}\Bigr]\mid X_i,Z_{i,j}, Z_{i,k}\Bigr]\Bigr]\\
&=\E\Bigl[\1_{\{|\tilde{H}(i,i)|<\frac{\tilde{\eps}\theta_n^2n}{2p}\}}\E\Bigl[\tilde{\Phi}^2(i,j)\mid X_i,Z_{i,j}\Bigr]\E\Bigl[\tilde{\Phi}^2(i,k)\mid X_i, Z_{i,k}\Bigr]\Bigr]\\
\intertext{and by \eqref{eq:tildeG}}
&=\E\bigl[\1_{\{|\tilde{H}(i,i)|<\frac{\tilde{\eps}\theta_n^2n}{2p}\}}\tilde{G}_j(i,i)\tilde{G}_k(i,i)\Bigr]
=\E\bigl[Z_{i,j}\bigr]\E\bigl[Z_{i,k}\bigr]\E\bigl[\1_{\{|\tilde{H}(i,i)|<\frac{\tilde{\eps}\theta_n^2n}{2p}\}}\tilde{H}(i,i)\tilde{H}(i,i)\bigr]
\\
&<p^2\cdot \frac{\tilde{\eps}\theta_n^2n}{2p}\cdot\EW{\tilde{H}(i,i)}
\leq p\cdot \frac{\tilde{\eps}\theta_n^2n}{2}\cdot\frac{\beta_n^2}{p}
=\frac{\tilde{\eps}\theta_n^2n}{2}\beta_n^2
\end{align*}
where in the last inequality we applied Corollary \ref{cor:expectG}.

On the other hand, for $A_{i2}$, we know that at least one of the values $j,k,l,m$ is different from the others. Without loss of generality, this is $m$.
Then,
\begin{align*}
\E\bigl[\1_{\{|\tilde{H}(i,i)|<\frac{\tilde{\eps}\theta_n^2n}{2p}\}}\tilde{\Phi}(i,j)\tilde{\Phi}(i,k)\tilde{\Phi}(i,l)\tilde{\Phi}(i,m)\bigr]\hspace{-8cm}&\hspace{8cm}\\
&=\E\bigl[\E\bigl[\1_{\{|\tilde{H}(i,i)|<\frac{\tilde{\eps}\theta_n^2n}{2p}\}}\tilde{\Phi}(i,j)\tilde{\Phi}(i,k)\tilde{\Phi}(i,l)\tilde{\Phi}(i,m)\mid Z_{i,j}, Z_{i,k}, Z_{i,l}, Z_{i,m}, X_i, X_j, X_k, X_l\bigr]\bigr]\\
&=\E\bigl[\1_{\{|\tilde{H}(i,i)|<\frac{\tilde{\eps}\theta_n^2n}{2p}\}}\tilde{\Phi}(i,j)\tilde{\Phi}(i,k)\tilde{\Phi}(i,l)Z_{i,m}\EW{\tilde{h}(i,m)\mid X_i}\Bigr]=0,
\end{align*}
due to \eqref{eq:htildecentering}.
Altogether,
$A_i^2=A_{i1}\leq n^2\cdot \frac{\tilde{\eps}\theta_n^2n}{2}\cdot\beta_n^2\leq n^2\cdot \tilde{\eps}n\theta_n^2\cdot\beta_n^2,$ hence
\begin{equation}
A_i\leq n\beta_n\theta_n\sqrt{\tilde{\eps}n},
\label{eq:aieq}
\end{equation}
To give a bound for $B_i$, we use the fact that
\[\E\Bigl[\bigl(\sum\limits_{l=1}^{i-1}\tilde{\Phi}(i,l)\bigr)^2\Bigl]=\E\Bigl[\sum\limits_{l=1}^{i-1}\tilde{\Phi}^2(i,l)\Bigr]+\E\Bigl[\sum\limits_{l,m=1,l\neq m}^{i-1}\tilde{\Phi}(i,l)\tilde{\Phi}(i,m)\Bigr]=\sum\limits_{l=1}^{i-1}\E\bigl[\tilde{\Phi}^2(i,l)\bigr]\leq n\beta_n^2,\]
by Lemma \ref{lem:phikleinindep} and Lemma \ref{lem:lem phi und h}. By Markov's inequality
\begin{equation}
B_i=\P\Big(\bigl|\sum\limits_{l=1}^{i-1}\tilde{\Phi}(i,l)\bigr|\geq\frac{\eps\theta_nn}{2}\Bigr)^{1/2}\leq\Bigl(\frac{4}{\eps^2\theta_n^2n^2}\E\bigr[\bigl(\sum\limits_{l=1}^{i-1}\tilde{\Phi}(i,l)\bigr)^2\bigr]\Bigr)^{1/2}\leq \frac{2\beta_n}{\eps\theta_n\sqrt{n}}.
\label{eq:bieq}
\end{equation}
Then \eqref{28}, \eqref{eq:aieq} and \eqref{eq:bieq} give
\begin{align}
S_{42}&\leq \frac{4}{n^2\theta_n^2}\sum\limits_{i=1}^{n}A_iB_i\notag\leq \frac{4}{n^2\theta_n^2}\sum\limits_{i=1}^{n}(n\beta_n\theta_n \sqrt{\tilde{\eps}n}\cdot\frac{2\beta_n}{\eps\theta_n\sqrt{n}})
\leq \frac{4}{n^2\theta_n^2}n\left(n\frac{2\sqrt{\tilde{\eps}}}{\eps}\beta_n^2\right)\leq \frac{16\sqrt{\tilde{\eps}}}{\eps}\notag
\end{align}
As $\tilde{\eps}$ was chosen arbitrarily, we obtain,
$S_{42}\nconv 0$, and hence $S_4\nconv 0.$

Wrapping things up, this tells us that
\[\EW{\eta_1}\leq S_1+T_1\leq S_1+S_2+S_3+S_4 \nconv 0\]
and thus, \eqref{16} holds.
\end{proof}

\begin{lem}
\eqref{17} follows from \eqref{6}, \eqref{20}, \eqref{21}, \eqref{23} and \eqref{22}.
\end{lem}
\begin{proof}
We start by calculating $\xi_i^2$:
\begin{align}
\xi_i^2&=\left(\xi_i^{(1)}+\xi_i^{(2)}\right)^2=\bigl(
\frac{1}{n\theta_n}\sum_{\substack{j=1\\j\neq i}}^{n}\Psi_j(i)+\frac{1}{n\theta_n}\sum\limits_{j=1}^{i-1}\tilde{\Phi}(i,j)\bigr)^2
\notag\\
&=\frac{1}{n^2\theta_n^2}\sum_{\substack{j=1\\j\neq i}}^{n}\Psi_j^2(i)+\frac{1}{n^2\theta_n^2}\sum_{\substack{j,k=1\\j,k\neq i, j\neq k}}^{n}\Psi_j(i)\Psi_k(i)+2\frac{1}{n^2\theta_n^2}\sum_{\substack{j=1\\j\neq i}}^{n}\sum\limits_{k=1}^{i-1}\Psi_j(i)\tilde{\Phi}(i,k)\notag\\
&\phantom{=E}+\sum\limits_{j=1}^{i-1}\tilde{\Phi}^2(i,j)+\frac{1}{n^2\theta_n^2}\sum\limits_{\substack{j,k=1\\j\neq k}}^{i-1}\tilde{\Phi}(i,j)\tilde{\Phi}(i,k)\label{many_summands}
\end{align}
We will compute the sum (in $i$) of the conditional expectations for of each of these summands.
Let us start with the second one:
First observe that for any choice of $i \ne j,k$ we have $\Psi_k(i)=\Psi_j(i)Z_{ik}/Z_{ij}$. Hence, using that $Z_{ij}^2=Z_{ij}$ we get
\begin{align*}
&\sum\limits_{i=1}^{n}\frac{1}{n^2\theta_n^2}\sum_{\substack{j,k=1\\j,k\neq i, j\neq k}}^{n}\EW{\Psi_j(i)\Psi_k(i)\mid(X_k)_{k=1,\dots,i-1},(Z_{l,m})_{\substack{l=1,\dots,i-1,\phantom{m\neq l}\\m=1,\dots,n,\,m\neq l}}}\\
&=\sum\limits_{i=1}^{n}\frac{1}{n^2\theta_n^2}\sum_{\substack{j,k=1\\j,k\neq i, j\neq k}}^{n}\EW{Z_{i,k}\Psi_j^2(i)\mid(X_k)_{k=1,\dots,i-1},(Z_{l,m})_{\substack{l=1,\dots,i-1,\phantom{m\neq l}\\m=1,\dots,n,\,m\neq l}}}\\
&=\sum\limits_{i=1}^{n}\frac{1}{n^2\theta_n^2}\sum_{\substack{j=1\\j\neq i}}^{n}\Bigl(\sum\limits_{\substack{k=1\\k\neq j}}^{i-1}Z_{i,k}+\sum\limits_{\substack{k=i+1\\k\neq i,j}}^{n}p\Bigr)\EW{\Psi_j^2(i)\mid(X_k)_{k=1,\dots,i-1},(Z_{l,m})_{\substack{l=1,\dots,i-1,\phantom{m\neq l}\\m=1,\dots,n,\,m\neq l}}}
\end{align*}
where we applied measurability of $Z_{i,k}$ for $k<i$ with respect to the condition and  independence of the condition for $Z_{i,k}$, $k>i$.

The conditional expectation of first two summands in \eqref{many_summands} together then is
\begin{align*}\eta_{21}&\coloneqq\sum\limits_{i=1}^{n}\frac{1}{n^2\theta_n^2}\sum_{\substack{j=1\\j\neq i}}^{n}\E\bigl[\Psi_j^2(i)\mid(X_k)_{k=1,\dots,i-1},(Z_{l,m})_{\substack{l=1,\dots,i-1,\phantom{nnl}\\m=1,\dots,n,\,m\neq l}}\bigr]\cdot\bigl(\sum\limits_{\substack{k=1\\k\neq j}}^{i-1}Z_{i,k}+\sum\limits_{\substack{k=i+1\\k\neq i,j}}^{n}p+1\bigr)\\
&=\sum\limits_{i=1}^{n}\frac{1}{n^2\theta_n^2}\sum_{\substack{j=1\\j\neq i}}^{n}\E\Bigl[Z_{ij}^2\E\bigl[h(i,j)\mid X_i\bigr]^2\mid(Z_{l,m})_{\substack{l=1,\dots,i-1,\phantom{nnl}\\m=1,\dots,n,\,m\neq l}}\Bigr]\cdot\bigl(\sum\limits_{\substack{k=1\\k\neq j}}^{i-1}Z_{i,k}+\sum\limits_{\substack{k=i+1\\k\neq i,j}}^{n}p+1\bigr)\\
&=\sum\limits_{i=1}^{n}\frac{1}{n^2\theta_n^2}\sum_{\substack{j=1\\j\neq i}}^{n}\E\Bigl[Z_{ij}^2\E\bigl[\E\bigl[h(i,j)\mid X_i\bigr]^2\bigr]\mid(Z_{l,m})_{\substack{l=1,\dots,i-1,\phantom{nnl}\\m=1,\dots,n,\,m\neq l}}\Bigr]\cdot\bigl(\sum\limits_{\substack{k=1\\k\neq j}}^{i-1}Z_{i,k}+\sum\limits_{\substack{k=i+1\\k\neq i,j}}^{n}p+1\bigr)\\
&=\frac{\gamma_n^2}{n^2p\theta_n^2}\sum\limits_{i=1}^{n}\sum_{\substack{j=1\\j\neq i}}^{n}\E\Bigl[Z_{i,j}\cdot\bigl(\sum\limits_{\substack{k=1\\k\neq j}}^{i-1}Z_{i,k}+\sum\limits_{\substack{k=i+1\\k\neq i,j}}^{n}p+1\bigr)\mid(Z_{l,i})_{l=1,\dots,i-1}\Bigr].
\end{align*}
since $\EW{\EW{h(i,j)\mid X_i}^2}$ is given by Lemma \ref{lem:lem phi und h}, and $Z_{i,j}^2=Z_{i,j}$.
It is easily seen that for this term, the relation
\begin{align}
\EW{\eta_{21}}
&=\frac{\gamma_n^2}{n^2p\theta_n^2}n(n-1)\left((n-2)p+1\right)p\sim \frac{np\gamma_n^2}{\theta_n^2}
\label{eq:expeta21}
\end{align}
holds.
Next, the conditional expectation of the fourth summand in \eqref{many_summands} can be computed as:
\begin{align*}\eta_{22}&:=\frac{1}{n^2\theta_n^2}\sum\limits_{i=1}^{n}\sum_{j=1}^{i-1}\E\Bigl[\tilde{\Phi}^2(i,j)\mid(X_k)_{k=1,\dots,i-1},(Z_{l,m})_{\substack{l=1,\dots,i-1,\phantom{nnl}\\m=1,\dots,n,\,m\neq l}}\Bigr]\\
&=\frac{1}{n^2\theta_n^2}\sum\limits_{i=1}^{n}\sum_{j=1}^{i-1}\EW{\tilde{\Phi}^2(i,j)\mid X_j,Z_{i,j}}=\frac{1}{n^2\theta_n^2}\sum\limits_{i=1}^{n}\sum_{j=1}^{i-1}\tilde{G}_i(j,j)
\end{align*}
by \eqref{eq:tildeG}. By Corollary \ref{cor:expectG} the expectation for this term satisfies
\begin{equation}
\EW{\eta_{22}}=\binom{n}{2}\frac{1}{n^2\theta_n^2}(\beta_n^2-2\gamma_n^2)\sim\frac{\beta_n^2-2\gamma_n^2}{2\theta_n^2}
\label{eq:expeta22}
\end{equation}
Furthermore, the sum of the conditional expectations of fifth summand in \eqref{many_summands} by \eqref{eq:tildeG} is
$ \eta_{23}
:=\frac{1}{n^2\theta_n^2}\sum\limits_{i=1}^{n}\sum\limits_{\substack{j,k=1\\j\neq k}}^{i-1}\tilde{G}_i(j,k).
$
For the third summand in \eqref{many_summands} we obtain:
\begin{align*}
\eta_{24}
&:=\frac{2}{n^2\theta_n^2}\sum\limits_{i=1}^{n}\sum_{\substack{j=1\\j\neq i}}^{n}\sum\limits_{m=1}^{i-1}\E[\Psi_j(i)\tilde{\Phi}(i,m)\mid X_m,(Z_{l,i})_{l=1,\dots,i-1,}]
\end{align*}
Then
$\eta_2\sim\eta_{21}+\eta_{22}+\eta_{23}+\eta_{24}.$ 
One can immediately conclude from \eqref{eq:expeta21} and \eqref{eq:expeta22} 
that
$\EW{\eta_{21}+\eta_{22}}\sim \frac{np\gamma_n^2}{\theta_n^2}+\frac{\beta_n^2-2\gamma_n^2}{2\theta_n^2}\sim\frac{\frac{1}{2}\beta_n^2+np\gamma_n^2}{\theta_n^2}$
and by definition of $\theta_n^2$ (cf. \eqref{defn})
\begin{equation}
\EW{\eta_{21}+\eta_{22}}\nconv 1.
\label{eq:conveta21-1}
\end{equation}
Let us split up $\eta_{22}$, by choosing $\tilde{\eps}>0$ arbitrarily:
\begin{align*}\eta_{22}
&=\frac{1}{n^2\theta_n^2}\sum\limits_{i=1}^{n}\sum\limits_{j=1}^{i-1}\tilde{G}_i(j,j)\1_{\bigl\{|\tilde{H}(j,j)|<\frac{\tilde{\eps}\theta_n^2n}{p}\bigr\}}+\frac{1}{n^2\theta_n^2}\sum\limits_{i=1}^{n}\sum\limits_{j=1}^{i-1}\tilde{G}_i(j,j)\1_{\bigl\{|\tilde{H}(j,j)|\geq\frac{\tilde{\eps}\theta_n^2n}{p}\bigr\}}\\
&=:\eta_{22}'+\eta_{22}''
\end{align*}
Consider the second summand first. By definition
\begin{align*}
\EW{\eta_{22}''}
&=\frac{1}{n^2\theta_n^2}\sum\limits_{i=1}^{n}\sum\limits_{j=1}^{i-1}\E\bigl[Z_{i,j}\tilde{H}(j,j)\1_{\bigl\{|\tilde{H}(j,j)|\geq\frac{\tilde{\eps}\theta_n^2n}{p}\bigr\}}\bigr]
=\frac{1}{n^2}\binom{n}{2}\frac{p}{\theta_n^2}\E\bigl[\tilde{H}(1,1)\1_{\bigl\{|\tilde{H}(1,1)|\geq\frac{\tilde{\eps}\theta_n^2n}{p}\bigr\}}\bigr].
\end{align*}
Then \eqref{21} yields for any choice of $\tilde{\eps}$ that $\EW{\eta_{22}''}\nconv 0$ and hence
$\eta_{22}''\Pconv 0$ and by \eqref{eq:conveta21-1}
\begin{equation}
\EW{\eta_{21}+\eta_{22}'}\nconv 1.
\label{eq:eta21+eta22''conv}
\end{equation}

To compute $\E[(\eta_{21}+\eta_{22}')^2]$ we define  $\Lambda_j(i)\coloneqq\sum\limits_{\substack{k=1\\k\neq j}}^{i-1}Z_{i,k}+\sum\limits_{\substack{k=i+1\\k\neq i,j}}^{n}p+1$ and compute
\begin{align*}
	\E\bigl[\eta_{21}^2\bigr]
=\frac{\gamma_n^4}{n^4p^2\theta_n^4}\sum\limits_{i,i'=1}^{n}\sum_{\substack{j=1\\j\neq i}}^{n}\sum_{\substack{j'=1\\j'\neq i'}}^{n}\mathbb{E}\Bigl[\E\bigl[\Lambda_j(i)Z_{i,j}\mid(Z_{l,i})_{l=1,\dots,i-1}\bigr]\cdot\E\bigl[\Lambda_{j'}(i')Z_{i',j'}\mid(Z_{l,i'})_{l=1,\dots,i'-1}\bigr]\Bigr]
	\end{align*}
 If $i=i', j=j'$, by Jensen's inequality we find
\begin{multline*}
\mathbb{E}\Bigl[\E\bigl[\Lambda_j(i)Z_{i,j}\mid(Z_{l,i})_{l=1,\dots,i-1}\bigr]
\cdot\E\bigl[\Lambda_{j'}(i')Z_{i',j'}\mid(Z_{l,i'})_{l=1,\dots,i'-1}\bigr]\Bigr]\\
\leq\EW{Z_{ij}^2\Lambda_j^2(i)}=\EW{Z_{ij}Z_{i'j'}\Lambda_j(i)\Lambda_{j'}(i')}
\end{multline*}	
If $i=i'$, $j\neq j'$, $\Lambda_j(i)$ and $\Lambda_{j'}(i)$ are $(Z_{l,i})_{l=1,\dots,i-1}$-measurable. Then, dragging the second conditional into the first one and additionally conditioning on $Z_{i,j}$ in it (which is possible due to $j\neq j'$),
\begin{align*}
\mathbb{E}&\Bigl[\E\bigl[\Lambda_j(i)Z_{i,j}\mid(Z_{l,i})_{l=1,\dots,i-1}\bigr]
\cdot\E\bigl[\Lambda_{j'}(i')Z_{i',j'}\mid(Z_{l,i'})_{l=1,\dots,i'-1}\bigr]\Bigr]\\
&=\mathbb{E}\Bigl[\Lambda_j(i)\Lambda_{j'}(i)\cdot \E\bigl[Z_{i,j}\mid(Z_{l,i})_{l=1,\dots,i-1}\bigr]
\E\bigl[Z_{i,j'}\mid(Z_{l,i})_{l=1,\dots,i-1}\bigr]\Bigr]\\
&=\mathbb{E}\Bigl[\Lambda_j(i)\Lambda_{j'}(i)\cdot\E\bigl[\E\bigl[Z_{i,j}Z_{i,j'}\mid(Z_{l,i})_{l=1,\dots,i-1,j}\bigr]\mid(Z_{l,i})_{l=1,\dots,i-1}\bigr]\Bigr]\\
&=\mathbb{E}\Bigl[Z_{i,j}Z_{i,j'}\cdot \Lambda_j(i)\Lambda_{j'}(i)\Bigr],
\end{align*}
where the last equality is due to the above measurability again and then applying law of total expectation.

Finally, if $i\neq i'$, by independence and law of total expectation we have
\begin{align*}
\mathbb{E}&\Bigl[\E\bigl[\Lambda_j(i)Z_{i,j}\mid(Z_{l,i})_{l=1,\dots,i-1}\bigr]
\cdot\E\bigl[\Lambda_{j'}(i')Z_{i',j'}\mid(Z_{l,i'})_{l=1,\dots,i'-1}\bigr]\Bigr]\\
&=\mathbb{E}\Bigl[\Lambda_j(i)Z_{i,j}\Bigr]\cdot\mathbb{E}\Bigl[\Lambda_{j'}(i')Z_{i',j'}\Bigr]=\mathbb{E}\Bigl[Z_{i,j}Z_{i,j'}\cdot \Lambda_j(i)\Lambda_{j'}(i)\Bigr].
\end{align*}
Thus
\[\EW{\eta_{21}^2}\leq\frac{\gamma_n^4}{n^4p^2\theta_n^4}\sum\limits_{i=1}^{n}\sum\limits_{i'=1}^{n}\sum_{j=1\atop j\neq i}^{n}\sum_{j'=1\atop j'\neq i'}^{n}\mathbb{E}\left[Z_{i,j}Z_{i',j'}\Lambda_j(i)\Lambda_{j'}(i')\right]\]

	By Lemma \ref{lem:mischtermkunterbunt2} this immediately leads to
		$\EW{\eta_{21}^2}\lesssim\frac{(np)^2\gamma_n^4}{\theta_n^4}.$
Moreover,
	\begin{align*}
	\EW{\eta_{21}\eta_{22}'}&=\mathbb{E}\Bigl[\frac{\gamma_n^2}{n^2p\theta_n^2}\sum\limits_{i=1}^{n}\sum_{\substack{j=1\\j\neq i}}^{n}\E\bigl[Z_{i,j}\Lambda_j(i)\mid(Z_{l,i})_{l=1,\dots,i-1}\bigr]
\frac{1}{n^2\theta_n^2}\sum\limits_{i'=1}^{n}\sum\limits_{j'=1}^{i'-1}\tilde{G}_{i'}(j',j')\1_{\left\{|\tilde{H}(j',j')|<\frac{\tilde{\eps}\theta_n^2n}{p}\right\}}\Bigr]\\
	&\leq\mathbb{E}\Bigl[\frac{\gamma_n^2}{n^2p\theta_n^2}\sum\limits_{i=1}^{n}\sum_{\substack{j=1\\j\neq i}}^{n}\E\bigl[Z_{i,j}\Lambda_j(i)\mid(Z_{l,i})_{l=1,\dots,i-1}\bigr]\cdot\frac{1}{n^2\theta_n^2}\sum\limits_{i'=1}^{n}\sum\limits_{j'=1}^{i'-1}\tilde{G}_{i'}(j',j')\Bigr]
	\end{align*}
Similarly to the previous step we get by \eqref{eq:tildeG} 
\begin{align*}
\EW{\eta_{21}\eta_{22}'} 
\leq\frac{\gamma_n^2}{n^4p\theta_n^4}\sum\limits_{i,i'=1}^{n}\sum_{\substack{j=1\\j\neq i}}^{n}\sum\limits_{j'=1}^{i'-1}\mathbb{E}\left[\Lambda_j(i)Z_{i,j}Z_{i',j'}\tilde{h}^2(i',j')\right]
	\end{align*}
Now the $\tilde{h}$ term only depends on the $X_{i'}$, such that is independent of $\Lambda_j(i)Z_{i,j}Z_{i',j'}$. By Lemma \ref{lem:lem phi und h}, $\EW{\tilde{h}^2(i,j)}=\frac{\beta_n^2-2\gamma_n^2}{p}\leq \frac{\beta_n^2}{p}$ such that
\begin{align*}
\EW{\eta_{21}\eta_{22}'} &
\leq\frac{\gamma_n^2\beta_n^2}{n^4p^2\theta_n^4}\sum\limits_{i=1}^{n}\sum_{\substack{j=1\\j\neq i}}^{n}\sum\limits_{i'=1}^{n}\sum\limits_{j'=1}^{i'-1}\mathbb{E}\left[\Lambda_j(i)Z_{i,j}Z_{i',j'}\right]
	\end{align*}
	Applying Lemma \ref{lem:mischtermkunterbunt3} yields
	$\EW{\eta_{21}\eta_{22}'}\lesssim\frac{np}{2\theta_n^4}\gamma_n^2\beta_n^2.$
	Finally,
	\begin{align*}
	\EW{\left(\eta_{22}'\right)^2}
	&=\frac{1}{n^4\theta_n^4}\sum\limits_{i=1}^{n}\sum\limits_{j=1}^{i-1}\sum\limits_{i'=1}^{n}\sum\limits_{j'=1}^{i'-1}\E\Bigl[\tilde{G}_{i}(j,j)\1_{\bigl\{|\tilde{H}(j,j)|<\frac{\tilde{\eps}\theta_n^2n}{p}\bigr\}}\tilde{G}_{i'}(j',j')\1_{\bigl\{|\tilde{H}(j',j')|<\frac{\tilde{\eps}\theta_n^2n}{p}\bigr\}}\Bigr]\\
&=\frac{1}{n^4\theta_n^4}\sum\limits_{i=1}^{n}\sum\limits_{j=1}^{i-1}\sum\limits_{i'=1}^{n}\E\Bigl[\tilde{G}_{i}(j,j)\1_{\bigl\{|\tilde{H}(j,j)|<\frac{\tilde{\eps}\theta_n^2n}{p}\bigr\}}\tilde{G}_{i'}(j,j)\1_{\bigl\{|\tilde{H}(j,j)|<\frac{\tilde{\eps}\theta_n^2n}{p}\bigr\}}\Bigr]\\
	&\hspace{1cm}+\frac{1}{n^4\theta_n^4}\sum\limits_{i=1}^{n}\sum\limits_{j=1}^{i-1}\sum\limits_{i'=1}^{n}\sum\limits_{\substack{j'=1\\j'\neq j}}^{i'-1}\E\Bigl[\tilde{G}_{i}(j,j)\1_{\bigl\{|\tilde{H}(j,j)|<\frac{\tilde{\eps}\theta_n^2n}{p}\bigr\}}\tilde{G}_{i'}(j',j')\1_{\bigl\{|\tilde{H}(j',j')|<\frac{\tilde{\eps}\theta_n^2n}{p}\bigr\}}\Bigr]\\
	\intertext{and applying \eqref{eq:tildeG} to both sums gives} 
	&\leq\frac{1}{n^4\theta_n^4}\sum\limits_{i=1}^{n}\sum\limits_{j=1}^{i-1}\sum\limits_{i'=1}^{n}\E\Bigl[Z_{i,j}Z_{i',j}\tilde{H}^2(j,j)\1_{\bigl\{|\tilde{H}(j,j)|<\frac{\tilde{\eps}\theta_n^2n}{p}\bigr\}}\Bigr]\\
	&\hspace{1cm}+\frac{1}{n^4\theta_n^4}\sum\limits_{i=1}^{n}\sum\limits_{j=1}^{i-1}\sum\limits_{i'=1}^{n}\sum\limits_{\substack{j'=1\\j'\neq j}}^{i'-1}\E\bigl[Z_{i,j}Z_{i',j'}\tilde{H}(j,j)\tilde{H}(j',j')\bigr]
	\end{align*}
By independence we arrive at
	\begin{align*}
&\EW{\left(\eta_{22}'\right)^2}	\le\frac{1}{n^4\theta_n^4}\sum\limits_{i=1}^{n}\sum\limits_{j=1}^{i-1}\sum\limits_{i'=1}^{n}
\E\bigl[Z_{i,j}Z_{i',j}\bigr]\E\bigl[\tilde{H}^2(j,j)
\1_{\left\{|\tilde{H}(j,j)|<\frac{\tilde{\eps}\theta_n^2n}{p}\right\}}\bigr]\\
&\hspace{1cm}+\frac{1}{n^4\theta_n^4}\sum\limits_{i=1}^{n}\sum\limits_{j=1}^{i-1}\sum\limits_{i'=1}^{n}\sum\limits_{\substack{j'=1\\j'\neq j}}^{i'-1}\EW{Z_{i,j}}\EW{Z_{i',j'}}\E\bigl[\tilde{H}(j,j)\bigr]\E\bigl[\tilde{H}(j',j')\bigr]\\	&\leq\frac{1}{n^4\theta_n^4}\sum\limits_{i=1}^{n}\sum\limits_{j=1}^{i-1}\sum\limits_{i'=1}^{n}\EW{Z_{i,j}Z_{i',j}}\frac{\tilde{\eps}\theta_n^2n}{p}\EW{\tilde{H}(j,j)}\\
	&\hspace{1cm}+\frac{1}{n^4\theta_n^4}\sum\limits_{i=1}^{n}\sum\limits_{j=1}^{i-1}\sum\limits_{i'=1}^{n}\sum\limits_{\substack{j'=1\\j'\neq j}}^{i'-1}\EW{Z_{i,j}}\EW{Z_{i',j'}}\E\bigl[\tilde{H}(j,j)\bigr]\E\bigl[\tilde{H}(j',j')\bigr]\\
	
&\leq\frac{1}{n^4\theta_n^4}\sum\limits_{i=1}^{n}\sum\limits_{j=1}^{i-1}\sum\limits_{i'=1}^{n}\EW{Z_{i,j}Z_{i',j}}\frac{\tilde{\eps}\theta_n^2n}{p}\frac{\beta_n^2}{p}
+\frac{1}{n^4\theta_n^4}\sum\limits_{i=1}^{n}\sum\limits_{j=1}^{i-1}\sum\limits_{i'=1}^{n}\sum\limits_{\substack{j'=1\\j'\neq j}}^{i'-1}\EW{Z_{i,j}}\EW{Z_{i',j'}}\left(\frac{\beta_n^2}{p}\right)^2
	\end{align*}
	where we applied Corollary \ref{cor:expectG} and used the bound $\beta_n^2-2\gamma_n^2\leq \beta_n^2$. 
 By $Z_{i,j}^2=Z_{i,j}$
\begin{align*}
	\EW{\left(\eta_{22}'\right)^2}
&\leq\frac{1}{n^4\theta_n^4}\frac{\tilde{\eps}\theta_n^2n}{p}\frac{\beta_n^2}{p}\left(\binom{n}{2}p+\binom{n}{2}(n-1)p^2\right)+\frac{1}{n^4\theta_n^4}\binom{n}{2}^2p^2\left(\frac{\beta_n^2}{p}\right)^2\\
&\sim\frac{1}{n^4\theta_n^4}\frac{\tilde{\eps}\theta_n^2n}{p}\frac{\beta_n^2}{p}\binom{n}{2}(n-1)p^2+\frac{1}{n^4\theta_n^4}\binom{n}{2}^2p^2\left(\frac{\beta_n^2}{p}\right)^2
	\sim\tilde{\eps}+\frac{\beta_n^4}{4\theta_n^4}
	\end{align*}
	by $\frac{\beta_n^2}{2}\leq\theta_n^2$. Thus, after a quick calculation
\begin{align*}
\EW{\left(\eta_{21}+\eta_{22}'\right)^2}
\lesssim \frac{(np)^2\gamma_n^4}{\theta_n^4} + 2 \frac{np}{2\theta_n^4}\gamma_n^2\beta_n^2+\frac{\beta_n^4}{4\theta_n^4}+\tilde{\eps}
&=\frac{1}{\theta_n^4}\left(np\gamma_n^2+\frac{1}{2}\beta_n^2\right)^2+\tilde\eps
=1+\tilde{\eps}
\label{eq:secondmomsumme}
\end{align*}

Putting this together with \eqref{eq:eta21+eta22''conv}, we obtain
$\V\bigl(\eta_{21}+\eta_{22}'\bigr)
\lesssim\tilde{\eps}+o(1).$

Hence $\eta_{21}+\eta_{22}'$ converges in probability to the limit of its expectation, which is 1. 
It remains to show that $\eta_{23},\eta_{24}\Pconv 0$, then
$\eta_{21}+\eta_{22}+\eta_{23}+\eta_{24}\Pconv 1.$
We start with $\eta_{23}$. By similar calculations as above:
\begin{align*}
\EW{\eta_{23}^2}
&=\frac{1}{n^4\theta_n^4}\sum\limits_{i=1}^{n}\sum\limits_{\substack{j,k=1\\j\neq k}}^{i-1}\sum\limits_{i'=1}^{n}\sum\limits_{\substack{j',k'=1\\j'\neq k'}}^{i'-1}\mathbb{E}\left[\tilde{G}_i(j,k)\tilde{G}_{i'}(j',k')\right]\\
&=\frac{1}{n^4\theta_n^4}\sum\limits_{i=1}^{n}\sum\limits_{\substack{j,k=1\\j\neq k}}^{i-1}\sum\limits_{i'=1}^{n}\sum\limits_{\substack{j',k'=1\\j'\neq k'}}^{i'-1}\EW{\tilde{\Phi}(i',j')\tilde{\Phi}(i',k')\tilde{\Phi}(i,j)\tilde{\Phi}(i,k)}
\end{align*}
which converges to 0 by Lemma \ref{lem:asymptgtilde}.
Thus
$\eta_{23}^2\Pconv 0$ follows.

Finally for $\eta_{24}$,
\begin{align*}
\EW{\eta_{24}^2}&=\frac{4}{n^4\theta_n^4}\sum\limits_{i=1}^{n}\sum_{\substack{j=1\\j\neq i}}^{n}\sum\limits_{m=1}^{i-1}\sum\limits_{i'=1}^{n}\sum_{\substack{j'=1\\j'\neq i'}}^{n}\sum\limits_{m'=1}^{i'-1}
\mathbb{E}\Bigl[\E\bigl[\Psi_j(i)\tilde{\Phi}(i,m)\mid X_m,(Z_{l,i})_{\substack{l=1,\dots,i-1}}\bigr]\Bigr.\\
&\hspace{-0.5cm}\phantom{\frac{4}{n^4\theta_n^4}\sum\limits_{i=1}^{n}\sum_{\substack{j=1\\j\neq i}}^{n}\sum\limits_{m=1}^{i-1}\sum\limits_{i'=1}^{n}\sum_{\substack{j'=1\\j'\neq i'}}^{n}\sum\limits_{m'=1}^{i'-1}\mathbb{E}}\cdot\Bigl.\E\bigr[\Psi_{j'}({i'})\tilde{\Phi}({i'},{m'})\mid X_{m'},(Z_{l,i'})_{\substack{l=1,\dots,i'-1}}\bigr]\Bigr]
\end{align*}
which converges to 0 by Lemma \ref{lem:eta24} and consequently
$\eta_{24}\Pconv 0.$
By the above estimates for the $\eta_{2i}, i=1, \ldots 4$ we conclude 
\[\eta_2\Pconv 1,\]
which completes the proof.
\end{proof}

\section{Alternative conditions for the Central Limit Theorem}
As mentioned above, it may be sometimes cumbersome to check the condition \eqref{6}--\eqref{21} in Theorem \ref{theo2}. We now give an alternative.

\begin{prop}\label{prop:altcond}
The conditions \eqref{6}-\eqref{21} follow from
\begin{align}
n^2\theta_n^{-2}\EW{\Psi_2^2(1)\1_{\{|\Psi_2(1)|\geq \eps\theta_n\}}}&\nconv 0\label{6neu}\tag{C1''}\\
\theta_n^{-2}\EW{\Phi^2(1,2)\1_{\left\{|\Phi(1,2)|\geq \eps\theta_n n\right\}}}&\nconv 0\label{20neu}\tag{C2''}\\
p\,\theta_n^{-2}\E\Bigl[H(1,1)\1_{\left\{|H(1,1)|\geq \frac{\eps\theta_n^2 n}{p}\right\}}\Bigr]&\nconv 0,\label{21neu}\tag{C3''}
\end{align}
for any $\eps>0$.
\end{prop}

\begin{proof}
We use the definitions of $\tilde{\Phi}$ and $\tilde{G}$:
\begin{enumerate}[leftmargin=*]
\item[\eqref{6}:] We use Lemma \ref{lem:quadratabsch} for $l=2,\dots,n$ and $a_l=\Psi_l(1)$. Then
\begin{align*}
\frac{1}{n\theta_n^2}\E\Bigl[\bigl(\sum\limits_{j=2}^{n}\Psi_j(1)\bigr)^2\1_{\bigl\{\bigl|\sum\limits_{j=2}^{n}\Psi_j(1)\bigr|\geq\eps\theta_n n\bigr\}}\Bigr]&\leq \frac{1}{n\theta_n^2}\E\Bigl[n^2\sum\limits_{j=2}^{n}\Psi_j^2(1)\1_{\{|\Psi_j(1)|\geq\eps\theta_n\}}\Bigr]\\
&\leq \frac{n^2}{\theta_n^2}\E\bigl[\Psi_2^2(1)\1_{\{|\Psi_2(1)|\geq\eps\theta_n\}}\bigr] \to 0
\end{align*}
by identical distribution and \eqref{6neu}. Therefore, \eqref{6} is true.
	\item[\eqref{20}:] By Lemma \ref{lem:quadratabsch} for $k=3$,
	\begin{align*}
	\theta_n^{-2}&\EW{\tilde{\Phi}^2(1,2)\1_{\left\{|\tilde{\Phi}(1,2)|\geq \eps\theta_n n\right\}}}\leq 9\theta_n^{-2}\EW{\Phi^2(1,2)\1_{\{|\Phi(1,2)|\geq\frac{\eps\theta_n n}{3}\}}}\\
	&\phantom{=E}+9\theta_n^{-2}\EW{\Psi_2^2(1)\1_{\left\{|\Psi_2(1)|\geq\frac{\eps\theta_n n}{3}\right\}}}+9\theta_n^{-2}\EW{\Psi_1^2(2)\1_{\left\{|\Psi_1(2)|\geq\frac{\eps\theta_n n}{3}\right\}}}\\
	&\leq 9\theta_n^{-2}\EW{\Phi^2(1,2)\1_{\{|\Phi(1,2)|\geq\frac{\eps\theta_n n}{3}\}}}\\
	&\phantom{=E}+9\theta_n^{-2}\EW{\Psi_2^2(1)\1_{\left\{|\Psi_2(1)|\geq\frac{\eps\theta_n}{3}\right\}}}+9\theta_n^{-2}\EW{\Psi_1^2(2)\1_{\left\{|\Psi_1(2)|\geq\frac{\eps\theta_n}{3}\right\}}}
	\end{align*}
	By \eqref{20neu}, the first term converges to 0. By \eqref{6neu}, so do the other two. Therefore, \eqref{20} is true.
	\item[\eqref{21}:]
	For $i\neq k$, we have
	\begin{align}
	\tilde{H}(i,i)&=\E\bigl[\tilde{h}(i,k)\tilde{h}(i,k)\mid X_i\bigr]\notag\\	
	&=\mathbb{E}\Bigl[h(i,k)h(i,k)-\EW{h(i,k)\mid X_i}h(i,k)-\EW{h(i,k)\mid X_k}h(i,k)\notag\\
	&\phantom{=E}\left.-\EW{h(i,k)\mid X_i}h(i,k)-\EW{h(i,k)\mid X_k}h(i,k)\right.\notag\\
	&\phantom{=E}\left.+\EW{h(i,k)\mid X_i}\EW{h(i,k)\mid X_i}+\EW{h(i,k)\mid X_i}\EW{h(i,k)\mid X_k}\right.\notag\\
	&\phantom{=E}+\EW{h(i,k)\mid X_k}\EW{h(i,k)\mid X_i}+\EW{h(i,k)\mid X_k}\EW{h(i,k)\mid X_k}\mid X_i\Bigr]\notag
\end{align}	
By measurability and independence we obtain after a short computation
\begin{align}
\tilde{H}(i,i)	
	&=H(i,i)-\EW{h(i,k)\mid X_i}^2-\EW{\EW{h(i,k)\mid X_k}h(i,k)\mid X_i}\notag\\
	&\phantom{=E}-\EW{h(i,k)\mid X_i}^2-\EW{\EW{h(i,k)\mid X_k}h(i,k)\mid X_i}\notag\\
	&\phantom{=E}+\EW{h(i,k)\mid X_i}^2+\EW{h(i,k)\mid X_i}\EW{\EW{h(i,k)\mid X_k}}\notag\\
	&\phantom{=E}+\EW{h(i,k)\mid X_i}\EW{\EW{h(i,k)\mid X_k}}+\EW{\EW{h(i,k)\mid X_k}^2}\notag\\
	\intertext{As $h(i,j)$ is centered, and by Lemma \ref{lem:lem phi und h}}
	&=H(i,i)-\EW{h(i,k)\mid X_i}^2-2\EW{\EW{h(i,k)\mid X_k}h(i,k)\mid X_i}+\frac{\gamma_n^2}{p}\notag\\
	&\eqqcolon H(i,i)-A-2B+\frac{\gamma_n^2}{p}.
	\end{align}

Firstly, by Lemma \ref{lem:lem phi und h}
	$\EW{A}=\EW{\EW{h(i,k)\mid X_k}^2}=\frac{\gamma_n^2}{p}$.
	
	By Cauchy-Schwarz and Lemma \ref{lem:lem phi und h}, we  obtain
	\begin{align*}\bigl|\E\bigl[2B\1_{\left\{|2B|\geq\frac{\eps\theta_n^2n}{5p}\right\}}\bigr]\bigr|&\leq2\sqrt{\EW{\EW{h(i,k)\mid X_k}^2}\E\bigl[h(i,k)^2\1_{\{|2B|\geq\frac{\eps\theta_n^2n}{5p}\}}\bigr]}\\
	&\leq2\sqrt{\EW{\EW{h(i,k)\mid X_k}^2}\E\bigl[h(i,k)^2\bigr]}\leq 2\frac{\gamma_n\beta_n}{p}\end{align*}
	We obtain by similar arguments as in Lemma \ref{lem:quadratabsch}
	\begin{align*}
	&p\,\theta_n^{-2}\E\Bigl[\tilde{H}(1,1)\1_{\bigl\{|\tilde{H}(1,1)|\geq \frac{\eps\theta_n^2 n}{p}\bigr\}}\Bigr]\\
	&\phantom{=}\leq5p\,\theta_n^{-2}\Bigl(\E\bigl[H(1,1)\1_{\bigl\{|H(1,1)|\geq\frac{\eps\theta_n^2n}{5p}\bigr\}}\bigr]+\Bigl|\E\bigl[A\1_{\left\{|A|\geq\frac{\eps\theta_n^2n}{5p}\right\}}\bigr]\Bigr|+\Bigl|\E\bigl[2B\1_{\left\{|2B|\geq\frac{\eps\theta_n^2n}{5p}\right\}}\bigr]\Bigr|+\frac{\gamma_n^2}{p}\Bigr)\\
	&\phantom{=}\leq5p\,\theta_n^{-2}\Bigl(\E\Bigl[H(1,1)\1_{\left\{|H(1,1)|\geq\frac{\eps\theta_n^2n}{5p}\right\}}\Bigr]+\left|\EW{A}\right|+2\frac{\gamma_n\beta_n}{p}+\frac{\gamma_n^2}{p}\Bigr)
\end{align*}
\begin{align*}	
	&\phantom{=}\leq5p\,\theta_n^{-2}\Bigl(\E\Bigl[H(1,1)\1_{\left\{|H(1,1)|\geq\frac{\eps\theta_n^2n}{5p}\right\}}\Bigr]+2\frac{\gamma_n^2}{p}+2\frac{\gamma_n\beta_n}{p}\Bigr)\\
	&\phantom{=}=5p\,\theta_n^{-2}\E\Bigl[H(1,1)\1_{\left\{|H(1,1)|\geq\frac{\eps\theta_n^2n}{5p}\right\}}\Bigr]+10p\,\theta_n^{-2}\frac{\gamma_n^2}{p}+10p\,\theta_n^{-2}\frac{\beta_n\gamma_n}{p}\\
	\end{align*}
	Since $\theta_n^2\geq np\gamma_n^2$ and $\theta_n^2\geq\frac{1}{2}\beta_n^2$, the last  two terms immediately converge to 0. By \eqref{21neu}, so does the first one. Therefore, \eqref{21} is true.
\end{enumerate}
This completes the proof.
\end{proof}

\begin{appendices}
\section{Appendix}
We start the appendix by proving the lemmas in the introduction.

\begin{proof}[Proof of Lemma \ref{lem:lem phi und h}]
We begin with (\ref{properties:phi})
\begin{align*}
\EW{\tilde{\Phi}^2(i,j)}&=\EW{\left(\Phi(i,j)-\Psi_j(i)-\Psi_i(j)\right)^2}\\
&=\mathbb{E}\left[\Phi^2(i,j)-2\Psi_j(i)\Phi(i,j)-2\Psi_i(j)\Phi(i,j)
+\Psi_j^2(i)+2\Psi_j(i)\Psi_i(j)+\Psi_i^2(j)\right]\\
&=\EW{\Phi^2(i,j)}-4\EW{\Psi_j(i)\Phi(i,j)}+2\EW{\Psi_j^2(i)}+2\EW{\Psi_j(i)\Psi_i(j)}
\end{align*}
due to identical distribution.
The last term is 0, since $\Psi_j(i)$ and $\Psi_i(j)$ are independent and centered. The first and third term are known from \eqref{defn}. As for the second term,
\begin{align*}
\EW{\Psi_j(i)\Phi(i,j)}&=\EW{\EW{\Phi(i,j)\mid X_i, Z_{ij}}\Phi(i,j)}\\&=\EW{\EW{\EW{\Phi(i,j)\mid X_i, Z_{ij}}\Phi(i,j)\mid X_i,Z_{ij}}}\\&=\EW{\EW{\Phi(i,j)\mid X_i,Z_{ij}}^2}=\EW{\Psi_j^2(i)}=\gamma_n^2,
\end{align*}
by measurability, so that
$\EW{\tilde{\Phi}^2(i,j)}=\beta_n^2-4\gamma_n^2+2\gamma_n^2=\beta_n^2-2\gamma_n^2.$

As for the statements on $h$ and $\tilde h$, i.e.\ (\ref{properties:h}):
We have
\[\gamma_n^2=\EW{\Psi_j^2(i)}=\EW{Z_{i,j}}\EW{\EW{h(i,j)\mid X_i}^2}=p\EW{\EW{h(i,j)\mid X_i}^2}\]
by definition of $\Psi_j(i)$ and independence. Moreover,
\[\beta_n^2=\EW{\Phi^2(i,j)}=\EW{Z_{i,j}h^2(i,j)}=\EW{Z_{i,j}}\EW{h^2(i,j)}=p\EW{h^2(i,j)}\]
by definition of $\Phi(i,j)$ and independence. Finally,
\[\beta_n^2-2\gamma_n^2=\EW{\tilde{\Phi}^2(i,j)}=\EW{Z_{i,j}\tilde{h}^2(i,j)}=\EW{Z_{i,j}}\EW{\tilde{h}^2(i,j)}=p\EW{\tilde{h}^2(i,j)}\]
by (\ref{properties:phi}) and independence.
This proves the claim.
\end{proof}

\begin{proof}[Proof of Lemma \ref{lem:hoeffing_dec}]
We have
\begin{align}
\CUn&=\binom{n}{2}^{-1}\sum\limits_{1\leq i<j\leq n}\tilde{\Phi}(i,j)+\binom{n}{2}^{-1}\sum\limits_{1\leq i<j\leq n}\Psi_j(i)+\binom{n}{2}^{-1}\sum\limits_{1\leq i<j\leq n}\Psi_i(j)\notag\\
&=\binom{n}{2}^{-1}\sum\limits_{1\leq i<j\leq n}\tilde{\Phi}(i,j)+\binom{n}{2}^{-1}\sum\limits_{1\leq i<j\leq n}\Psi_j(i)+\binom{n}{2}^{-1}\sum\limits_{1\leq j<i\leq n}\Psi_j(i)\notag\\
&=\binom{n}{2}^{-1}\Bigl(\sum\limits_{1\leq i<j\leq n}\tilde{\Phi}(i,j)+\sum_{i=1}^{n}\sum_{\substack{j=1\\j\neq i}}^{n}\Psi_j(i)\Bigr)\notag
=\binom{n}{2}^{-1}\Bigl(\sum\limits_{1\leq j<i\leq n}\tilde{\Phi}(i,j)+\sum_{i=1}^{n}\sum_{\substack{j=1\\j\neq i}}^{n}\Psi_j(i)\Bigr).\notag
\end{align}
\end{proof}

\begin{proof}[Proof of Lemma \ref{lem:variance}]
Since $\CUn$ ist centered, we obtain from \eqref{eq:hoeffdingdecomposition}
\begin{align*}
\V{\CUn}&=\EW{\CUn^2}\\
&=\binom{n}{2}^{-2}\E\Bigl[\sum\limits_{i<j}\tilde{\Phi}^2(i,j)\Bigr]+\binom{n}{2}^{-2}\mathbb{E} \Bigl[\sum\limits_{i<j}\sum\limits_{\substack{k<l\\\{i,j\}\neq\{k,l\}}}\tilde{\Phi}(i,j)\tilde{\Phi}(k,l)\Bigr]\\
&\phantom{=E}+2\binom{n}{2}^{-2}\EW{\sum\limits_{i<j}\tilde{\Phi}(i,j)\sum\limits_{k\neq l}\Psi_l(k)}\\
&\phantom{=E}+\binom{n}{2}^{-2}\mathbb{E}\Bigl[\sum\limits_{j\neq i}\Psi_j^2(i)\Bigr]
+\binom{n}{2}^{-2}\E\Bigl[\sum\limits_{j\neq i}\sum\limits_{\substack{l\neq k\\\{i,j\}\neq\{k,l\}}}\Psi_j(i)\Psi_l(k)\Bigr]\\
&=:A+B+C+D+E
\end{align*}
Let us consider the summands separately: Note that
\begin{equation}
A=\binom{n}{2}^{-1}\EW{\tilde{\Phi}^2(1,2)}=\binom{n}{2}^{-1}\left(\beta_n^2-2\gamma_n^2\right).
\label{eq:A}
\end{equation}
Moreover, $B=C=0$ as follows from Lemma \ref{lem:phikleinindep}.
For $D$ notice that
\begin{equation}
D=\binom{n}{2}^{-1}2\gamma_n^2
\label{eq:D}
\end{equation}
Finally, consider $E$. For $k\neq i$, the expectation is 0 (see the arguments given in the proof of Lemma \ref{lem:phikleinindep}).
For $k=i$, we have that $j\neq l$ and therefore
\begin{align*}
\EW{\Psi_j(i)\Psi_l(k)}&=\EW{Z_{ij}Z_{il}\EW{h(i,j)\mid X_i}\EW{h(i,l)\mid X_i}}=\EW{Z_{ij}Z_{il}\EW{h(i,j)\mid X_i}^2}\\
&=\EW{Z_{il}\Psi_j^2(i)} =\EW{Z_{il}}\EW{\Psi_j^2(i)}=p\gamma_n^2
\end{align*}
Thus
\begin{equation}
E=\binom{n}{2}^{-1}\cdot2(n-2)p\gamma_n^2
\label{eq:E}
\end{equation}
and
\begin{align*}
\V{\CUn}
=\binom{n}{2}^{-1}\left(\beta_n^2+2(n-2)p\gamma_n^2\right)
\sim\binom{n}{2}^{-1}\left(\beta_n^2+2np\gamma_n^2\right)= \binom{n}{2}^{-1}2\theta_n^2,
\end{align*}
from which we conclude the assertion.
\end{proof}

We now prove a couple of lemmas that were used in the proof of Theorem \ref{theo2} in Section 4.

\begin{cor}\label{cor:expectG}
For any $i\neq j$
$\EW{\tilde{G}_j(i,i)}=\beta_n^2-2\gamma_n^2$
and
$\EW{\tilde{H}(i,i)}=\frac{\beta_n^2-2\gamma_n^2}{p}.$
\end{cor}
\begin{proof}
The claim follows immediately from the tower property, the definition of $\tilde{G}$ and Lemma \ref{lem:lem phi und h}:
\[\EW{\tilde{G}_j(i,i)}=\EW{\EW{\tilde{\Phi}(i,j)\tilde{\Phi}(i,j)\mid X_i,Z_{i,j}}}=\EW{\tilde{\Phi}^2(i,j)}=\beta_n^2-2\gamma_n^2.\]
For $\tilde{H}$, one can use
\[p\EW{\tilde{H}(i,i)}=\EW{Z_{i,j}}\EW{\tilde{H}(i,i)}=\EW{Z_{i,j}\tilde{H}(i,i)}=\EW{\tilde{G}_j(i,i)}\]
and apply the above result.
\end{proof}

\begin{lem}\label{lem:phikleinindep}
For $\{i,j\}\neq\{k,l\}$ we have
$\EW{\tilde{\Phi}(i,j)\tilde{\Phi}(k,l)}=0.$

For any $\{i,j\},\{k,l\}$ we have
$\EW{\tilde{\Phi}(i,j)\Psi_l(k)}=0.$
\end{lem}

\begin{proof}
Consider two cases:

If $\{i,j\}\cap\{k,l\}=1$, then by identical distributions and \eqref{eq:htildecentering}
\begin{align*}\EW{\tilde{\Phi}(1,2)\tilde{\Phi}(1,3)}
&=\EW{Z_{12}Z_{13}}\EW{\tilde{h}(1,2)\tilde{h}(1,3)}\\
&=\EW{Z_{12}Z_{13}}\EW{\EW{\tilde{h}(1,2)\tilde{h}(1,3)\mid X_1,X_2}}\\
&=\EW{Z_{12}Z_{13}}\EW{\tilde{h}(1,2)\EW{\tilde{h}(1,3)\mid X_1,X_2}}\\
&=\EW{Z_{12}Z_{13}}\EW{\tilde{h}(1,2)\EW{\tilde{h}(1,3)\mid X_1}}=0.
\end{align*}

If $\{i,j\}\cap\{k,l\}=0$, then by similar reasoning
\begin{align*}\EW{\tilde{h}(1,2)\tilde{h}(3,4)}
&=0.\end{align*}
This shows the first claim.
The second can be shown in the same fashion.
\end{proof}

\begin{lem}\label{lem:Gammalem}
For $\Lambda_j(i)$ the following relation holds
\[\EW{\Lambda_j^2(i)}=O((np)^2)\]
\end{lem}

\begin{proof}
With
$\Lambda_j(i)\coloneqq\sum\limits_{\substack{k=1\\k\neq j}}^{i-1}Z_{i,k}+\sum\limits_{\substack{k=i+1\\k\neq j}}^{n}p+1$ we can immediately conclude
	\begin{align*}\Lambda_j^2(i)&\leq 
	9\cdot \Bigl[\Bigl(\sum\limits_{k=1}^{i-1}Z_{i,k}\Bigr)^2+\bigl((n-i)p\bigr)^2+1\Bigr]\\
	&=9\cdot \Bigl[\sum\limits_{k=1}^{i-1}Z_{i,k}+\sum\limits_{k=1}^{i-1}\sum\limits_{\substack{l=1\\l\neq k}}^{i-1}Z_{i,k}Z_{i,l}+\bigl((n-i)p\bigr)^2+1\Bigr]
	\end{align*}
	Therefore, using the independence of the $Z_{i,j}$
	\begin{align*}
	\EW{\Lambda_j^2(i)}
	&\leq 9\left[\sum\limits_{k=1}^{i-1}p+\sum\limits_{k=1}^{i-1}\sum\limits_{\substack{l=1\\l\neq k}}^{i-1}p^2+((n-i)p)^2+1\right]\\
	&\leq 9\left[np+n^2p^2+n^2p^2+1\right],
	\end{align*}
	which, due to $np\to\infty$, confirms $\EW{\Lambda_j^2(i)}\leq O\left((np)^2\right).$
\end{proof}

\begin{lem}\label{lem:mischtermkunterbunt2}
With $\Lambda_j(i)$ as in Section 4 we have:
\begin{align*}
\sum\limits_{i=1}^{n}\sum\limits_{i'=1}^{n}\sum_{\substack{j=1\\j\neq i}}^{n}\sum_{\substack{j'=1\\j'\neq i'}}^{n}\EW{Z_{i,j}Z_{i',j'}\Lambda_j(i)\Lambda_{j'}(i')}\leq n^4p^2(np)^2
	\end{align*}
\end{lem}

\begin{proof}
	Recall that since $np\to\infty$
	\begin{equation}
	\Lambda_j(i)\coloneqq\sum\limits_{\substack{k=1\\k\neq j}}^{i-1}Z_{i,k}+\sum\limits_{\substack{k=i+1\\k\neq j}}^{n}p+1 \quad\text{ and } \quad
	\EW{\Lambda_j(i)}=(n-2)p+1\sim np.
	\label{eq:Gammaexp}
	\end{equation}
	
Let us diffentiate cases.\\
	If $i\neq i'$ (and we're not in the case $i=j'$, $j=i'$, which will be considered later), then $(Z_{i,k})_{\substack{k=1,\dots,i-1\\k\neq j}}, Z_{i,j}, (Z_{i',k})_{\substack{k=1,\dots,i'-1\\k\neq j'}}$ and $Z_{i',j'}$ are independent (regardless of $j$ and $j'$). Then by independence the expectation can be reduced to $\EW{\Lambda_j(i)}^2\EW{Z_{ij}}^2\sim (np)^2p^2$
	by \eqref{eq:Gammaexp}. There are $n\cdot (n-1)\cdot((n-1)\cdot (n-1)-1)$ possibilities for this case.\\
	If $i=j'$ and $j=i'$, the independence between $\Lambda_j(i), \Lambda_i(j)$ and $Z_{i,j}=Z_{j,i}=Z_{i',j'}$ still holds, as well as the independence between $\Lambda_j(i)$ and $\Lambda_i(j)$ and we obtain
	\begin{align*}
		\EW{\Lambda_j(i)\Lambda_{j'}(i')Z_{i,j}Z_{i',j'}}&=\EW{\Lambda_j(i)}\EW{\Lambda_{i}(j)}\EW{Z_{i,j}^2}
		\sim (np)^2p
	\end{align*}
	There are $n(n-1)$ possibilities for this case.\\
	If $i=i'$ but $j\neq j'$, then $Z_{i,j}$ may appear in the random sum in $\Lambda_{j'}(i)$ (and correspondingly, if we interchange $j,j'$).
We introduce
\begin{equation}
\Lambda_{j,j'}(i)\eqqcolon \sum\limits_{\substack{k=1\\k\neq j,j'}}^{i-1}Z_{i,k}+\sum\limits_{\substack{k=i+1\\k\neq j}}^{n}p+1,\quad\text{with } \EW{\Lambda_{j,j'}(i)}=(n-3)p+1\sim np
\label{eq:Lambdajj'}
\end{equation}
and $\Lambda_j(i)=\Lambda_{j,j'}(i)+Z_{i,j'}$
Then
\begin{align*}&\EW{\Lambda_j(i)\Lambda_{j'}(i')Z_{i,j}Z_{i',j'}}=\EW{\Lambda_j(i)\Lambda_{j'}(i)Z_{i,j}Z_{i',j'}}\\
&=\EW{\Lambda^2_{j,j'}(i)Z_{i,j}Z_{i,j'}}+\EW{Z_{i,j}Z_{i,j'}}+2\EW{\Lambda_{j,j'}(i)Z_{i,j}Z_{i,j'}}\\
&=\EW{\Lambda^2_{j,j'}(i)}p^2+p^2+2\EW{\Lambda_{j,j'}(i)}p^2.
\end{align*}
After some considerations one finds
\[\EW{\Lambda^2_{j,j'}(i)}=O\left((np)^2\right),\]
so that combining this and \eqref{eq:Lambdajj'} gives
\[\EW{\Lambda_j(i)\Lambda_{j'}(i')Z_{i,j}Z_{i',j'}}\leq O\left((np)^2\right)p^2+p^2+2O\left(np\right)p^2=O\left((np)^2\right)\cdot p^2\]
by $np\to\infty$.
There are $n(n-1)(n-2)$ possibilities for the case $i=i',j\neq j'$.\\
	%
	%
	If $i=i'$ and $j=j'$, we may again use independence to arrive at
		\begin{align*}
		\EW{\Lambda_j(i)\Lambda_{j'}(i')Z_{i,j}Z_{i',j'}}&=\EW{\Lambda_j^2(i)Z_{i,j}^2}
	\leq O\left((np)^2\right)p,
	\end{align*}
	by Lemma \ref{lem:Gammalem}.
	Again, there are $n(n-1)$ possibilities for this case.
	Putting this together, we see that the sum of all expectations is asymptotically bounded from above by
$n^4(np)^2p^2$.
	\end{proof}

\begin{lem}\label{lem:mischtermkunterbunt3}
The following relation holds:
\[\sum\limits_{i=1}^{n}\sum_{\substack{j=1\\j\neq i}}^{n}\sum\limits_{i'=1}^{n}\sum\limits_{j'=1}^{i'-1}\mathbb{E}\left[\Lambda_j(i)Z_{i,j}Z_{i',j'}\right]\leq \frac{1}{2}n^4p^2(np)\]
\end{lem}

\begin{proof}
The strategy of proof is exactly the same as in Lemma \ref{lem:mischtermkunterbunt2}, just the expectations and the number of summands differ. We therefore leave the proof to the reader.
\end{proof}

\begin{lem}\label{lem:asymptgtilde}
Under the assumptions of Theorem \ref{theo2} as $n \to \infty$ we have:
\begin{align*}
\frac{1}{n^4\theta_n^4}\sum\limits_{i=1}^{n}\sum\limits_{\substack{j,k=1\\j\neq k}}^{i-1}\sum\limits_{i'=1}^{n}\sum\limits_{\substack{j',k'=1\\j'\neq k'}}^{i'-1}\EW{\tilde{\Phi}(i',j')\tilde{\Phi}(i',k')\tilde{\Phi}(i,j)\tilde{\Phi}(i,k)}\nconv 0.
\end{align*}
\end{lem}

\begin{proof}
Without loss of generality, assume $i'\geq i$.
We let
\begin{align*}
Q&\coloneqq \mathbb{E}\left[\tilde{G}_i(j,k)\tilde{G}_{i'}(j',k')\right], \quad
\tilde{Q}\coloneqq\EW{\tilde{\Phi}(i,j)\tilde{\Phi}(i,k)\tilde{\Phi}(i',j')\tilde{\Phi}(i',k')}
\end{align*}
and
$Q_1\cdot Q_2\coloneqq\EW{Z_{i,j}Z_{i,k}Z_{i',j'}Z_{i',k'}}\EW{\tilde{h}(i,j)\tilde{h}(i,k)\tilde{h}(i',j')\tilde{h}(i',k')}$.
Note that all three notations denote the same object. However, we will use all these notations throughout the proof.

Now, let us go through all possible cases for $i,j,k,i',j',k'$.
\begin{enumerate}[leftmargin=*]
	\item If  $i=i'$ and $|\{j,k\}\cap\{j',k'\}|=2$ by independence $Q=\EW{\tilde{G}_i^2(j,k)}$.
	\item The cases  $i=i'$ and $|\{j,k\}\cap\{j',k'\}|=1$ and $i=i'$ and $|\{j,k\}\cap\{j',k'\}|=0$ are almost identical. Consider the first: without loss of generality take $j=j'$. Then, by total expectation, the tower property, and independence
	\begin{align*}
		Q_2
		&=\EW{\tilde{h}(i,j)^2\tilde{h}(i,k)\tilde{h}(i,k')}
		=\EW{\EW{\tilde{h}(i,j)^2\tilde{h}(i,k)\tilde{h}(i,k')\mid X_{i},X_{j},X_{k}}}\\
		&=\EW{\tilde{h}(i,j)^2\tilde{h}(i,k)\EW{\tilde{h}(i,k')\mid X_{i},X_{j},X_{k}}}
		=\EW{\tilde{h}(i,j)^2\tilde{h}(i,k)\EW{\tilde{h}(i,k')\mid X_{i}}}=0,
		\end{align*}
		since by \eqref{eq:htildecentering} $\EW{\tilde{h}(l,m)\mid X_r}=0$ if $l\neq m$ for every $r$. Thus,
		$Q=0$.
	\item Again the cases $i<i'$, $i\in\{j',k'\}$, and $|\{j,k\}\cap\{j',k'\}|=1$ and $i<i'$, $i\in\{j',k'\}$, and $|\{j,k\}\cap\{j',k'\}|=0$ are very similar. Consider the first: Without loss of generality $i=j',k=k'$, and along the lines of the previous cases we get $Q_2=0$. 	
	\item Next consider the case $i<i'$, $i\notin\{j',k'\}$, and $|\{j,k\}\cap\{j',k'\}|=2$. Without loss of generality, $j=j',k=k'$
	and by the definition of $\tilde{G}$ and independence we compute
			\begin{align*}
			Q
			&=\E\bigl[\tilde{G}_i(j,k)\tilde{G}_{i'}(j,k)\bigr]\\
			&=\mathbb{E}\bigl[\E\bigl[Z_{i,j}Z_{i,k}\tilde{h}(i,j)\tilde{h}(i,k)\mid X_j,X_k,Z_{i,j}Z_{i,k}\bigr]
	\E\bigl[Z_{i',j}Z_{i',k}\tilde{h}(i',j)\tilde{h}(i',k)\mid X_j,X_k,Z_{i',j}Z_{i',k}\bigr]\bigr]\\
			&=\E\bigl[Z_{i,j}Z_{i,k}Z_{i',j}Z_{i',k}\E\bigr[\tilde{h}(i,j)\tilde{h}(i,k)\mid X_j,X_k\bigr]^2\bigr]\\
			&=\E\bigl[Z_{i',j}Z_{i',k}\bigl(Z_{i,j}Z_{i,k}\E\bigl[\tilde{h}(i,j)\tilde{h}(i,k)\mid X_j,X_k,Z_{i,j}Z_{i,k}\bigr]\bigr)^2\bigr]\\
			&=\E\bigl[Z_{i',j}Z_{i',k}\bigl(\tilde{G}_i(j,k)\bigr)^2\bigr]
			=p^2\E\bigl[\tilde{G}_i^2(j,k)\bigr]
			\end{align*}
			\item Finally, the cases the case $i<i'$, $i\notin\{j',k'\}$, and $|\{j,k\}\cap\{j',k'\}|=0,1$ follow the arguments in cases (2) and (3) to give $Q_2=0$.
			\end{enumerate}
			
To sum up, what we get from this case distinction: 			
The only situation where the given expectation is non-zero is when $|\{j,k\}\cap\{j',k'\}|=2$.
In the case $i=i'$, there are at most $n(n-1)^2$ possibilities for this ($n$ for $i$, and since $j$ and $k$ are smaller than $i$ and different, at most $n-1$ for each of those).
In the case $i\neq i'$, there are an additional $n-1$ possibilities for $i'$, which makes at most $n(n-1)^3$ possibilities.
Altogether, we have that the given sum of expectations is bounded by
\[n(n-1)^2\E\Bigl[\bigl(\tilde{G}_i(j,k)\bigr)^2\Bigr]+n(n-1)^3p^2\E\Bigl[\bigl(\tilde{G}_i(j,k)\bigr)^2\Bigr]\leq n^4\E\bigl[\tilde{G}_i^2(j,k)\bigr].\]
Then for the sum of the considered expectations we have
\begin{align*}
\frac{1}{n^4\theta_n^4}&\sum\limits_{i=1}^{n}\sum\limits_{\substack{j,k=1\\j\neq k}}^{i-1}\sum\limits_{i'=1}^{n}\sum\limits_{\substack{j',k'=1\\j'\neq k'}}^{i'-1}\EW{\tilde{\Phi}(i',j')\tilde{\Phi}(i',k')\tilde{\Phi}(i,j)\tilde{\Phi}(i,k)}\leq\frac{\EW{\tilde{G}_i^2(j,k)}}{\theta_n^4}
\end{align*}
By \eqref{22}, this converges to 0.
\end{proof}

\begin{lem}\label{lem:eta24}
	Under the assumptions of Theorem \ref{theo2} as $n\to\infty$ we have
	\begin{align*}\frac{1}{n^4\theta_n^4}&\sum\limits_{i=1}^{n}\sum_{\substack{j=1\\j\neq i}}^{n}\sum\limits_{m=1}^{i-1}\sum\limits_{i'=1}^{n}\sum_{\substack{j'=1\\j'\neq i'}}^{n}\sum\limits_{m'=1}^{i'-1}
	\mathbb{E}\Bigl[\E\bigl[\Psi_j(i)\tilde{\Phi}(i,m)\mid X_m,(Z_{l,i})_{\substack{l=1,\dots,i-1}}\bigr]\Bigr.\\[-4ex]
	&\hspace{-0.5cm}\phantom{\frac{1}{n^4\theta_n^4}\sum\limits_{i=1}^{n}\sum_{\substack{j=1\\j\neq i}}^{n}\sum\limits_{m=1}^{i-1}\sum\limits_{i'=1}^{n}\sum_{\substack{j'=1\\j'\neq i'}}^{n}\sum\limits_{m'=1}^{i'-1}\mathbb{E}}\cdot\Bigl.\E\bigr[\Psi_{j'}({i'})\tilde{\Phi}({i'},{m'})\mid X_{m'},(Z_{l,i'})_{\substack{l=1,\dots,i'-1}}\bigr]\Bigr]\nconv 0.\end{align*}
\end{lem}
\begin{proof}
We denote
{\small{
		\[Q:=\EW{\Psi_j(i)\tilde{\Phi}(i,m)\mid X_m, (Z_{l,i})_{l=1,\dots,i-1}} \EW{\Psi_j(i')\tilde{\Phi}(i',m')\mid X_{m'}, (Z_{l,i'})_{l=1,\dots,i'-1}}=Q_1\cdot Q_2,
		\]}}
where
\begin{align*}
Q_1&=\EW{Z_{i,j}Z_{i,m}\mid  (Z_{l,i})_{l=1,\dots,i-1}}\cdot\EW{Z_{i',j'}Z_{i',m'}\mid  (Z_{l,i'})_{l=1,\dots,i'-1}}\\
Q_2&=\EW{\EW{h(i,j)\mid X_i}\tilde{h}(i,m)\mid X_m}\EW{\EW{h(i',j')\mid X_{i'}}\tilde{h}(i',m')\mid X_{m'}}.
\end{align*}
By independence between the $Z$- and $X$-terms, $\EW{Q}=\EW{Q_1}\EW{Q_2}$. \\

In the case $m\neq m'$, we have $\EW{Q_2}=0$ by independence of $X_m$ and $X_m'$ and $\EW{\EW{h(i,j)\mid X_i}\tilde{h}(i,m)}=0$, which we find by adding a conditional expectation on $X_i$.

For $m=m'$, note that the conditional expectations in $Q_2$ do not depend on the choice of $i$ and $i'$, hence we choose $i=1, i'=2$, and $m=3$.
Then:
\begin{align*}
\EW{Q_2}&=\EW{\EW{\EW{h(1,j)\mid X_1}\tilde{h}(1,3)\mid X_3}\EW{\EW{h(2,j')\mid X_2}\tilde{h}(2,3)\mid X_3}}\\
&=\EW{\EW{h(1,j)\mid X_1}\EW{h(2,j')\mid X_2}\tilde{H}_3(1,2)}.
\end{align*}
By Cauchy-Schwarz, independence and Lemma \ref{lem:lem phi und h}
\[\EW{Q_2}\leq\left(\E\Bigl[\EW{h(1,j)\mid X_1}^2\Bigr]\EW{\EW{h(2,j')\mid X_2}^2}\EW{\tilde{H}_3^2(1,2)}\right)^{1/2}=\frac{\gamma_n^2}{p}\EW{\tilde{H}_3^2(1,2)}^{1/2}\]

Furthermore, if  $i=i'$ and $|\{j,j',m\}|\leq 2$, $\EW{Q_1}\leq p$ and we have a at most $n^3$ possibilities to choose $i,j,m,i',j',m'$. \\

If $i=i'$ and $|\{j,j',m\}|= 3$, $\EW{Q_1}\leq p^3$ and at most $n^4$ possibilities to choose.\\

If $i\neq i'$ and $\{j,j',m\}|\leq 2$, $\EW{Q_1}\leq p^2$ and we have at most $n^4$ possibilities to choose.\\

If $i\neq i'$ and $\{j,j',m\}|= 3$, $\EW{Q_1}\leq p^4$ and we have at most $n^5$ possibilities to choose.

Combining all this and keeping in mind that we assume that $np \to \infty$ yields
\begin{align*}
\frac{1}{n^4\theta_n^4}&\sum\limits_{i=1}^{n}\sum_{\substack{j=1\\j\neq i}}^{n}\sum\limits_{m=1}^{i-1}\sum\limits_{i'=1}^{n}\sum_{\substack{j'=1\\j'\neq i'}}^{n}\sum\limits_{m'=1}^{i'-1}
\mathbb{E}\Bigl[\E\bigl[\Psi_j(i)\tilde{\Phi}(i,m)\mid X_m,(Z_{l,i})_{\substack{l=1,\dots,i-1}}\bigr]\Bigr.\\[-4ex]
&\hspace{-0.5cm}\phantom{\frac{1}{n^4\theta_n^4}\sum\limits_{i=1}^{n}\sum_{\substack{j=1\\j\neq i}}^{n}\sum\limits_{m=1}^{i-1}\sum\limits_{i'=1}^{n}\sum_{\substack{j'=1\\j'\neq i'}}^{n}\sum\limits_{m'=1}^{i'-1}\mathbb{E}}\cdot\Bigl.\E\bigr[\Psi_{j'}({i'})\tilde{\Phi}({i'},{m'})\mid X_{m'},(Z_{l,i'})_{\substack{l=1,\dots,i'-1}}\bigr]\Bigr]\\[-2ex]
&\leq\frac{1}{n^4\theta_n^4}\frac{\gamma_n^2}{p}\sqrt{\EW{\tilde{H}_1^2(2,3)}}\left(pn^3+p^3n^4+p^2n^4+p^4n^5\right)\\
&\leq\frac{2}{n^4\theta_n^4}\frac{\gamma_n^2}{p}\sqrt{\EW{\tilde{H}_1^2(2,3)}}\left(p^2n^4+p^4n^5\right)\\
\end{align*}
\begin{align*}
&\leq\frac{2}{np}\sqrt{\frac{1}{\theta_n^4}p^2\EW{\tilde{H}_1^2(2,3)}}+8p\sqrt{\frac{1}{\theta_n^4}p^2\EW{\tilde{H}_1^2(2,3)}}\\
&\leq \frac{2}{np}\sqrt{\frac{1}{\theta_n^4}\EW{\tilde{G}_1^2(2,3)}}+8p\sqrt{\frac{1}{\theta_n^4}\EW{\tilde{G}_1^2(2,3)}}.
\end{align*}
where the last two inequalities follow from \eqref{defn} and \eqref{eq:tildeG}.
By $np\to\infty$ and \eqref{22}, this converges to 0.
\end{proof}	

A very general lemma for the purpose of reminding us of a basic fact is
\begin{lem}\label{lem:quadratabsch}
For any sequence $(a_l)_{l=1,\dots,k}$, the following relation holds
\[\bigl(\sum\limits_{l=1}^ka_l\bigr)^2\1_{\bigl\{\bigl|\sum\limits_{l=1}^ka_l\bigr|\geq \eps\bigr\}}\leq  k^2\sum\limits_{l=1}^ka_l^2\1_{\bigl\{|a_l|\geq\frac{\eps}{k}\bigr\}}\]
\end{lem}
The proof of this lemma is elementary.
\end{appendices}

%
%
\end{document}